\documentclass[11pt,reqno]{amsart}

\usepackage[colorlinks, linktocpage, citecolor = purple, linkcolor = blue]{hyperref}
\usepackage[utf8]{inputenc}
\usepackage[margin=2cm]{geometry}
\usepackage{xparse}
\usepackage[dvipsnames]{xcolor}
\usepackage{changepage} 
\usepackage[inline]{enumitem}
\usepackage{verbatim}
\setlist[enumerate,1]{label = (\alph*),
                      ref = (\alph*)}
\usepackage{amsmath,amsthm,amssymb}
\usepackage{bm}
\usepackage{bbm}
\usepackage{braket}
\usepackage{mathtools}
\usepackage{MnSymbol}
\usepackage{mathdots}
\usepackage{mathrsfs}
\usepackage[mathscr]{euscript}

\let\mathscr\relax 
\usepackage[scr]{rsfso}
\usepackage{graphicx}
\usepackage{tikz}
\usetikzlibrary{arrows, cd, calc}
\usepackage{pgfplots}
\usepackage{ytableau}
\usepackage[numbers]{natbib}
\usepackage{hyperref}
\usepackage{booktabs}
\usepackage{multirow}
\usepackage{float}
\usepackage{varwidth}
\usepackage[font=small,labelfont=bf]{caption}
\usepackage{thmtools}
\usepackage{thm-restate}
\usepackage[capitalise]{cleveref}
\usepackage{subcaption}

\newcommand*{\N}{\mathbb{N}}
\newcommand*{\Z}{\mathbb{Z}}

\newcommand*{\PP}{\mathbb{P}}

\newcommand*{\cR}{\mathcal{R}}

\newcommand*{\abs}[1]{{\lvert #1 \rvert}}

\newcommand*{\given}{\mid}
\newcommand*{\maps}{\nobreak\mskip2mu\mathpunct{}\nonscript
  \mkern-\thinmuskip{:}\mskip6muplus1mu\relax}

\newcommand*{\card}[1]{\abs{#1}}
\newcommand*{\ti}[1]{\tilde{#1}}
\newcommand*{\wti}[1]{\widetilde{#1}}
\newcommand*{\restrict}[1]{{\mid}_{#1}}

\makeatletter
\let\@@pmod\pmod
\DeclareRobustCommand{\pmod}{\@ifstar\@pmods\@@pmod}
\def\@pmods#1{\mkern4mu({\operator@font mod}\mkern 6mu#1)}
\makeatother
\newcommand*{\disp}{\operatorname{disp}^{+}}

\newcommand*{\define}[1]{\textit{#1}}

\DeclareMathOperator{\Jac}{Jac}
\DeclareMathOperator{\codim}{codim}

\newcommand{\DeclareAutoPairedDelimiter}[3]{%
  \expandafter\DeclarePairedDelimiter\csname
  Auto\string#1\endcsname{#2}{#3}%
  \begingroup\edef\x{\endgroup
    \noexpand\DeclareRobustCommand{\noexpand#1}{%
      \expandafter\noexpand\csname Auto\string#1\endcsname*}}%
  \x }
\DeclareAutoPairedDelimiter{\p}{(}{)}
\DeclarePairedDelimiter{\ceil}{\lceil}{\rceil}

\theoremstyle{definition}
\newtheorem{definition}{Definition}[section]
\newtheoremstyle{problem}%
{}{}%
{}{}%
{\bfseries}{.}%
{ }%
{\thmname{#1}\thmnumber{ #2}\thmnote{ #3}}
\theoremstyle{problem}
\newtheorem{example}[definition]{Example}
\theoremstyle{plain}
\newtheorem{proposition}[definition]{Proposition}

\newtheorem{conjecture}[definition]{Conjecture}
\newtheorem{lemma}[definition]{Lemma}
\newtheorem{corollary}[definition]{Corollary}
\theoremstyle{remark}
\newtheorem{remark}[definition]{Remark}

\theoremstyle{theorem}
\newtheorem{maintheorem}{Theorem}	

\numberwithin{equation}{section}
\numberwithin{figure}{section}

\definecolor{c1}{RGB}{180,180,255}
\definecolor{c2}{RGB}{255,150,150}
\definecolor{c3}{RGB}{220,150,255}

\title[Prym--Brill--Noether loci of special curves]{Prym--Brill--Noether loci of special curves%
  \thanks{Research conducted at the Georgia Institute of Technology
    with the support of RTG grant GR10004614 and REU grant
      GR10004803} }

\author{Steven Creech}
\address{Department of Mathematics, Brown University, Providence, RI 02912, USA}
\email{\href{mailto:steven\_creech@brown.edu}{steven\_creech@brown.edu}}

\author{Yoav Len}
\address{School of Mathematics and Statistics, University of St Andrews,  St Andrews, KY16 9SS, UK}
\email{\href{mailto:yoav.len@st-andrews.ac.uk}{yoav.len@st-andrews.ac.uk}}

\author{Caelan Ritter}
\address{Department of Mathematics, University of Washington, Seattle, WA 98195, USA}
\email{\href{mailto:critter1@uw.edu}{critter1@uw.edu}}

\author{Derek Wu}
\address{School of Mathematics, Georgia Institue of
  Technology, Atlanta, GA 30332-0160, USA}
\email{\href{mailto:dwu96@gatech.edu}{dwu96@gatech.edu}}

\begin{document}

\begin{abstract}
  We use Young tableaux to compute the dimension of $V^r$, the
  Prym--Brill--Noether locus of a folded chain of loops of any
  gonality. This tropical result yields a new upper bound on the
  dimensions of algebraic Prym--Brill--Noether loci.  Moreover, we
  prove that $V^r$ is pure-dimensional and connected in codimension
  $1$ when $\dim V^r \geq 1$. We then compute the first Betti number of
  this locus for even gonality when the dimension is exactly $1$, and
  compute the cardinality when the locus is finite and the edge
  lengths are generic.
\end{abstract}
	
\maketitle
	
\setcounter{tocdepth}{1}
\tableofcontents


\section{Introduction}
Constructing algebraic cycles in abelian varieties can in general be a
challenging problem.  The problem becomes more tractable if we
restrict our attention to varieties that show up in the context of
algebraic curves such as Jacobians or Prym varieties. In such cases,
algebraic cycles may naturally be constructed by appealing to
Brill--Noether theory, and taking advantage of the geometric
interpretation of the points of the abelian varieties.

Let $f\maps\wti{X}\to X$ be an unramified double cover of either
tropical or algebraic curves, and let $f_*$ be the induced map on
divisor classes. The corresponding \define{Prym--Brill--Noether locus}
is
\begin{equation*}
  V^r(X,f) = \set{ [D]\in\Jac(\wti{X}) \given f_*(D) = K_X,\,
    r(D)\geq r,\, r(D)\equiv r \pmod* 2},
\end{equation*}
where $K_X$ is the canonical divisor of $X$.  It is a variation of the
usual Brill--Noether locus $W_r^d(\wti{X})$ that also takes
symmetries of $\wti{X}$ into account. The Prym--Brill--Noether locus
naturally lives inside the Prym variety associated with $f$
(see \cref{sec:preliminaries} for more details).  
Moreover, since the locus may be described as intersections of translates of the theta class, it is, in particular, tautological \cite[Theorem 1.2]{Arap_Prym}. 
This paper is concerned with the dimension and additional topological properties of $V^r(X,f)$. 

Properties of the usual Brill--Noether
loci  have been studied extensively for curves that are general in moduli  in 
classical algebraic geometry  \cite{GH, Gieseker_Petri, Fulton_Lazarsfeld_degeneracy} and more recently in
tropical geometry \cite{CDPR, JP14, Len1}.  When a curve is \emph{not} general in moduli, its Brill--Noether locus is no longer expected to be irreducible or pure-dimensional. 
Nevertheless, the dimensions of  irreducible components of these loci have recently been computed for general $k$-gonal curves, namely general among curves that admit a $k$-fold cover of $\PP^1$
\cite{Cook_Jensen_BN_Components, JR, Larson}. 


In contrast, much less is known for Prym varieties. Bertram and
Welters computed the dimension of the Prym--Brill--Noether locus for
curves that are general in  moduli \cite{Bertram_Prym,
  Welters_Prym}, and Welters has also shown that the locus is
generically smooth.  The tropical study of Prym varieties was
initially introduced in joint work of the second author with Jensen
\cite{JL}, and further studied in joint work with Ulirsch
\cite{len2019skeletons}.
As they show, tropical Pryms are abelian of the
expected dimension and behave well with respect to tropicalization,
leading to a new bound on the dimension of 
Prym--Brill--Noether loci of general even-gonal algebraic curves.


Our first result is an extension of these techniques to curves of
\emph{any} gonality.

\begin{restatable}{maintheorem}{tropicalPBN}
  \label{thm:tropicalPBN}%
  Let $\varphi\maps\wti\Gamma\to\Gamma$ be a $k$-gonal uniform folded
  chain of loops and denote by $l$ the quantity $\ceil{\frac{k}{2}}$.
  Then the codimension of $V^r(\Gamma,\varphi)$ relative to the Prym
  variety is given by
  \begin{equation}\label{eq:4}
    n(r,k) =
    \begin{cases}
      \binom{l+1}{2}+l(r-l) & \text{if $l \leq r-1$}\\
      \binom{r+1}{2} & \text{if $l>r-1$}
    \end{cases}.
  \end{equation}
\end{restatable}

\noindent By \emph{uniform} $k$-gonal we mean that the ratio of the
lengths of the upper and lower arcs of each loop is exactly $k$; see
\cref{sec:preliminaries} for more details.  We adopt the convention
that a set whose dimension is negative is empty, so
$V^r(\Gamma,\varphi)$ is empty if $n(r,k) > g - 1$.  As it turns out,
the odd gonality case is far trickier than the even, necessitating the
development of several new combinatorial tools.

As a consequence of the theorem, we obtain an
upper bound on the dimensions of Prym--Brill--Noether loci for
algebraic curves that are general in the $k$-gonal locus. In what
follows, we work over a non-Archimedean field $K$ with residue field
$\kappa$ whose characteristic is prime to both $2$ and
$k$. 

\begin{restatable}{maincorollary}{algebraicPBN}
  \label{cor:algebraicPBN}%
  Let $r\geq -1$ and $k\geq 2$.    Then there is a nonempty open subset
  of the $k$-gonal locus of $\cR_g$ such that for every unramified
  double cover $f\maps\wti{C}\to C$ in this open subset we have
  \begin{equation}\label{eq:1}
    \dim V^r(C,f) \leq g-1-n(r,k).
  \end{equation}
\end{restatable}




We then turn our attention to more subtle \emph{tropological}
properties of Prym--Brill--Noether loci of folded chains of
loops.\footnote{We introduce the descriptor ``tropological'' to mean
  ``topological'' in the context of tropical varieties.} 

\begin{maintheorem}\label{thm:pure-dim}
  $V^r(\Gamma,\varphi)$ is pure-dimensional for any gonality $k$. If
  $\dim V^r(\Gamma,\varphi)\geq 1$ then it is also connected in
  codimension $1$.
\end{maintheorem}

\noindent By ``connected in codimension $1$,'' we mean that any two maximal cells are
connected by a sequence of cells whose codimension relative to the locus is at most
$1$.  The different properties mentioned in the theorem are proved in Propositions
\ref{thm:pure-dim} and \ref{thm:path}.  The pure-dimensionality of the locus is quite
surprising since Brill--Noether loci of general $k$-gonal curves may very well have
maximal components of different dimension (see for instance \cite[Section 1]{JR}).
We do not know at this point whether this phenomenon is special to
tropical Prym curves or carries on to algebraic ones as well.

If we choose $r$ and $k$ so that $n(r,k) = g - 1$, the
Prym--Brill--Noether locus is a finite collection of points.  If $k$
is also assumed to be even, we may compute the cardinality by
constructing a bijection between its points and certain lattice paths
(\cref{prop:card}).
If the dimension is $1$, the tropical Prym--Brill--Noether locus is a
graph; we compute its first Betti number in the case of generic edge
lengths. 

\begin{restatable}{maintheorem}{genericdimone}
  \label{thm:generic-dim-1}
  Let $\varphi\maps\wti\Gamma\to\Gamma$ be a folded chain of loops with generic edge length
  such that $\dim V^r(\Gamma,\varphi)=1$. Then the first Betti number of
  $V^r(\Gamma,\varphi)$ is given by
  \begin{equation}\label{eq:9}
    \frac{r \cdot C(r,0) \cdot \left(\binom{r+1}{2}+1\right)}{2} + 1.
  \end{equation}
\end{restatable}
\noindent Here, $C(r,0)$ is the number of distinct ways to fill a
staircase tableau of size $r$ such that each symbol in the set
$[\binom{r+1}{2}]$ is used exactly once.\footnote{We define $C(r,k)$
  more generally in \cref{sec:finite} once we have more tools at our
  disposal.} Moreover, we calculate the first Betti number in the cases
where $k$ is 2 or 4 (Propositions \ref{prop:k2dim1} and
\ref{prop:k4dim1}).

Many of our results build on the correspondence between certain Young
tableaux and divisors on tropical curves (cf. \cite{CDPR,
  pflueger2017special}). The key tool that we develop to enumerate
such tableaux is the notion of a \emph{non-repeating strip}, a special
subset that determines the rest of the tableau (see
\cref{sec:an-interl}).  We hope that this and other techniques
presented in our paper will lead to additional results concerning
dimensions and Euler characteristics of tropical and algebraic
Brill--Noether loci.

There are numerous interesting avenues for investigation moving
forward. The techniques developed in \cite{JR} for lifting special
divisors should be adapted to the current situation to determine the
precise dimension of algebraic Prym--Brill--Noether loci (see
\cref{conj:dimension} for more details). If every maximal cell of
$V^r(\Gamma,\varphi)$ can be lifted, \cref{thm:pure-dim} would
moreover imply that algebraic Prym--Brill--Noether loci are
pure-dimensional.  It would be intriguing to extend the enumerative
results of \cref{sec:counting} to any gonality, and discover whether
an algebraic version holds as well.  Note, however, that current
degeneration techniques do not immediately imply either an upper or a
lower bound on the Betti numbers of algebraic Prym--Brill--Noether
curves. Finally, it would be exciting to extend our techniques to
ramified double covers and general Galois covers. The latter would be
especially challenging since the components of the kernel of such
covers do not naturally admit a principal polarization.

\subsection*{Acknowledgements} 
We thank Dave Jensen for helpful remarks on a previous version of
this manuscript.  We also thank the nameless referees for their insightful comments and suggestions. 
This research was conducted at the Georgia Institute
of Technology with the support of RTG grant GR10004614 and REU grant
GR10004803. 
    
\section{Preliminaries}\label{sec:preliminaries}
Throughout this paper, we use the terms metric graph and tropical
curve interchangeably.  We assume that the reader is familiar with the
theory of divisors on tropical curves; a beautiful introduction to
this topic may be found in \cite[Section 2]{HMY}.  Throughout, the
\define{genus} of a graph refers to its first Betti number, which also
equals one more than the number of edges minus the number of vertices.

The result of tropicalizing a covering map of algebraic curves is a \define{harmonic} morphism of metric graphs. Such morphisms  induce natural pushforward and pullback maps between divisors that respect the equivalence relation given by chip-firing.  A map of graphs  is called a \define{double cover} when it is harmonic of degree $2$, and \define{unramified} if, in
addition, it pulls back the canonical divisor of $\Gamma$ to the
canonical divisor of $\widetilde\Gamma$. See  \cite[Definition 2.7]{LUZ_Abelian_covers} for precise definitions of harmonic morphisms and their degree.


Fix a divisor class $[D]$ on $\Gamma$. The fiber
$\varphi^{-1}_{*}([D])$ consists of either one or two connected
components in the Picard group of $\widetilde\Gamma$ \cite[Proposition
6.1]{JL}. Each of them is referred to as a \define{Prym variety}, and
their elements are called \define{Prym divisor classes}.  Prym
varieties are principally polarized tropical abelian varieties
\cite[Theorem 2.3.7]{len2019skeletons}.  We take the divisor $D$
above to be the canonical divisor $K_\Gamma$.  
Fixing an integer $r$,
the \define{Prym--Brill--Noether locus} $V^r(\Gamma,\varphi)$ consists
of the Prym divisors whose rank is at least $r$ and has the same
parity as $r$.

Here we are interested in a particular double cover known as the
\define{folded chain of loops}. In this case, the target $\Gamma$ of the
map $\varphi$ is the \define{chain of loops} that recently appeared in
various celebrated papers (e.g. \cite{MRC, Pflueger, JR}). It consists
of $g$ loops, denoted by $\gamma_1,\ldots,\gamma_g$ and connected by
bridges. The source graph $\wti\Gamma$ is a chain of $2g-1$ loops, as
exemplified in \cref{fig:2}.  Each pair of loops $\ti\gamma_a$
and $\ti\gamma_{2g-a}$ (for $a<g$) maps down to $\gamma_a$, while each
edge of $\ti\gamma_g$ maps isometrically onto the loop $\gamma_g$.
See \cite[Section 5.2]{len2019skeletons} for a more detailed
explanation.

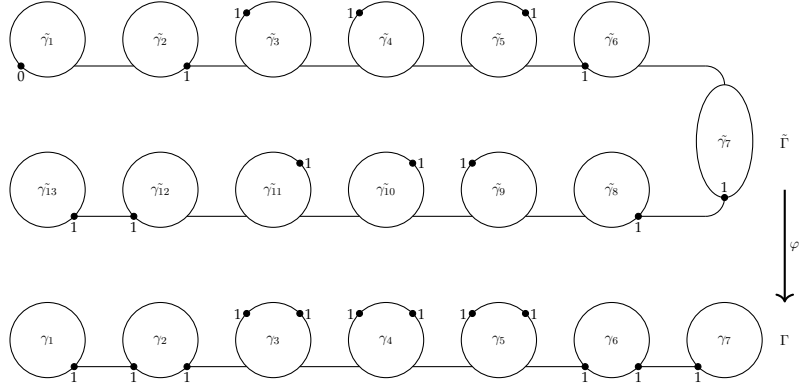
\begin{figure}[htb]
  \centering
  \begin{tikzpicture}[scale=2,every node/.style={scale=.5}]
  \draw (0,0) circle (.25cm)node {$\gamma_1$}; \draw
  (0.25*cos{315},0.25*sin{315})--(.75+.25*cos{225},.25*sin{225});
  \filldraw(0+.25*cos{315},.25*sin{315}) circle (0.02) node [below]{1};    
  \draw (.75,0) circle (.25)node {$\gamma_2$}; \draw
  (.75+.25*cos{315},0.25*sin{315})--(1.5+.25*cos{225},.25*sin{225});
  \filldraw (.75+.25*cos{225},.25*sin{225}) circle (.02) node [below]{1};
  \filldraw (.75+.25*cos{315},.25*sin{315}) circle (.02) node [below]{1};
  \draw (1.5,0) circle (.25)node {$\gamma_3$}; 
  \draw (1.5+0.25*cos{315},0.25*sin{315})--(2.25+.25*cos{225},.25*sin{225});
  \filldraw (1.5+.25*cos{135},.25*sin{135})circle (.02) node[left]{1}; 
  \filldraw (1.5+.25*cos{45},.25*sin{45}) circle (.02)node[right]{1}; 
  \draw (2.25,0) circle (.25)node {$\gamma_4$}; 
  \draw (2.25+0.25*cos{315},0.25*sin{315})--(3+.25*cos{225},.25*sin{225});
  \filldraw (2.25+.25*cos{135},.25*sin{135}) circle (.02) node [left]{1};
  \filldraw (2.25+.25*cos{45},.25*sin{45}) circle (.02) node [right]{1};
  \draw (3,0) circle (.25)node {$\gamma_5$}; 
  \draw (3+0.25*cos{315},0.25*sin{315})--(3.75+.25*cos{225},.25*sin{225});
  \filldraw (3+.25*cos{135},.25*sin{135}) circle (.02) node[left]{1};
  \filldraw (3+.25*cos{45},.25*sin{45})circle (.02) node[right]{1};
  \draw (3.75,0) circle (.25)node {$\gamma_6$}; 
  \draw (3.75+0.25*cos{315},0.25*sin{315})--(4.5+.25*cos{225},.25*sin{225});
  \filldraw (3.75+.25*cos{225},.25*sin{225}) circle (.02) node  [below]{1};
  \filldraw (3.75+.25*cos{315},.25*sin{315})  circle (.02) node [below]{1};
  \draw (4.5,0) circle (.25) node {$\gamma_7$};
  \filldraw (4.5+.25*cos{225},.25*sin{225}) circle (.02) node [below]{1};
  \draw  (4.9,0) node {$\Gamma$};
  \draw (0,1) circle (.25cm)node {$\tilde{\gamma_{13}}$}; 
  \draw (0.25*cos{315},1+0.25*sin{315})--(.75+.25*cos{225},1+.25*sin{225});
  \filldraw (0.25*cos{315},1+0.25*sin{315}) circle (.02cm) node  [below]{1};
  \draw (.75,1) circle (.25)node {$\tilde{\gamma_{12}}$}; 
  \draw (.75+.25*cos{315},1+0.25*sin{315})--(1.5+.25*cos{225},1+.25*sin{225});
  \filldraw (.75+.25*cos{225},1+.25*sin{225}) circle (.02) node  [below]{1};
  \draw (1.5,1) circle (.25)node {$\tilde{\gamma_{11}}$}; 
  \draw (1.5+0.25*cos{315},1+0.25*sin{315})--(2.25+.25*cos{225},1+.25*sin{225});
  \filldraw (1.5+.25*cos{45},1+.25*sin{45}) circle (.02) node  [right]{1}; 
  \draw (2.25,1) circle (.25)node {$\tilde{\gamma_{10}}$};
  \draw (2.25+0.25*cos{315},1+0.25*sin{315})--(3+.25*cos{225},1+.25*sin{225});
  \filldraw(2.25+.25*cos{45},1+.25*sin{45}) circle (.02) node  [right]{1};
  \draw (3,1) circle (.25)node {$\tilde{\gamma_9}$}; 
  \draw (3+0.25*cos{315},1+0.25*sin{315})--(3.75+.25*cos{225},1+.25*sin{225});
  \filldraw(3+.25*cos{135},1+.25*sin{135}) circle (.02) node  [left]{1};
  \draw (3.75,1) circle (.25)node {$\tilde{\gamma_8}$};
  \filldraw (3.75+.25*cos{315},1+.25*sin{315}) circle (.02) node  [below]{1};
  \draw (0,2) circle (.25cm)node {$\tilde{\gamma_1}$}; 
  \draw (0.25*cos{315},2+0.25*sin{315})--(.75+.25*cos{225},2+.25*sin{225});
  \filldraw (0+.25*cos{225},2+.25*sin{225}) circle (.02) node [below]{0};
  \draw (.75,2) circle (.25)node {$\tilde{\gamma_2}$}; 
  \draw  (.75+.25*cos{315},2+0.25*sin{315})--(1.5+.25*cos{225},2+.25*sin{225});
  \filldraw (.75+.25*cos{315},2+.25*sin{315}) circle (.02) node  [below]{1};
  \draw (1.5,2) circle (.25)node {$\tilde{\gamma_3}$}; 
  \draw  (1.5+0.25*cos{315},2+0.25*sin{315})--(2.25+.25*cos{225},2+.25*sin{225});
  \filldraw (1.5+.25*cos{135},2+.25*sin{135}) circle (.02)node [left]{1}; 
  \draw (2.25,2) circle (.25)node {$\tilde{\gamma_4}$}; 
  \draw  (2.25+0.25*cos{315},2+0.25*sin{315})--(3+.25*cos{225},2+.25*sin{225});
  \filldraw (2.25+.25*cos{135},2+.25*sin{135}) circle (.02) node [left]{1};
  \draw (3,2) circle (.25) node {$\tilde{\gamma_5}$}; 
  \draw  (3+0.25*cos{315},2+0.25*sin{315})--(3.75+.25*cos{225},2+.25*sin{225});
  \filldraw (3+.25*cos{45},2+.25*sin{45}) circle (.02) node [right]{1};
  \draw (3.75,2) circle (.25) node {$\tilde{\gamma_6}$};
  \filldraw (3.75+.25*cos{225},2+.25*sin{225}) circle (.02) node[below]{1};
  \draw (4.5,1.5+.25*sin{315}) ellipse (.1875 and .375)node
  {$\tilde{\gamma_7}$};
  \draw (4.5,1.5+.375+.25*sin{315} )%
  arc (0:90:2+.25*sin{315} -1.5-.375-.25*sin{315} );
  \draw (4.5,1.5-.375+.25*sin{315} )%
  arc (0:-90:2+.25*sin{315} -(1.5+.375+.25*sin{315} );
  \draw (3.75+.25*cos{315},2+.25*sin{315})%
  --(4.5-2-.25*sin{315}+1.5+.375+.25*sin{315},2+.25*sin{315});
  \draw (3.75+.25*cos{315},1+.25*sin{315})%
  --(4.5-2-.25*sin{315}+1.5+.375+.25*sin{315},1+.25*sin{315});
  \filldraw (4.5,1.5+.25*sin{315}-.375) circle (.02) node [above]{1};   
  \draw (4.9,1.5+.25*sin{315}) node {$\tilde\Gamma$};
  \draw[thick, ->] (4.9, 1) -- (4.9,.25);
  \draw (4.9,.625) node [right]{$\varphi$};
\end{tikzpicture}
  \caption{A Prym divisor on the 4-gonal folded chain of 13
   loops and
    its image under $\varphi_{*}$ on the 4-gonal chain of 7 loops.}
  \label{fig:2}
\end{figure}

The \define{torsion} of a loop $\gamma_a$ is the least positive
integer $k$ such that $\ell_a+m_a$ divides $k\cdot m_a$, where $m_a$
and $\ell_a$ are the lengths of the lower and upper arcs of $\gamma_a$
respectively.  The chain of loops is \define{uniform $k$-gonal} if each
loop has torsion $k$. Note that a uniform $k$-gonal chain of loops is
indeed a $k$-gonal metric graph in the sense of \cite[Section
1.3.2]{ABBR152}.  We say that a double cover as above is uniform
$k$-gonal if $\Gamma$ is, but note that $\wti\Gamma$ is not in itself
uniform $k$-gonal since the loop $\ti\gamma_g$ has torsion $2$.

\subsection{Prym tableaux}\label{sec:prym-tableaux}
We study divisors only indirectly, making use of a correspondence
between sets of divisors on chains of loops and Young tableaux as
introduced in \cite{pflueger2017special, len2019skeletons}; here we shall recall
only the essential definitions and introduce some helpful notation.

Let $[n] = \set{1,2,\ldots,n}$.  Given points
$(x,y), (x',y') \in \N^2$, which we call \define{boxes}, we say that
$(x,y)$ is \define{below} $(x',y')$ if $x \leq x'$, $y \leq y'$, and
$(x,y) \neq (x',y')$.  For our purposes, a \define{tableau} on a
subset $\lambda \subset \N^2$ is a map $t \maps \lambda \to [n]$
satisfying the \define{tableau condition}: 
\begin{quotation}
  for all boxes $(x,y)$ and $(x',y')$ in $\lambda$, if $(x,y)$ is
  below $(x',y')$, then $t(x,y) < t(x',y')$.
\end{quotation} We refer to the elements in the codomain of $t$ as
\define{symbols}.  Observe that if $\lambda$ is a partition of $n$ and
$t$ is injective, then $t$ is a standard Young tableau in the usual
sense.

A tableau $t$ is a \define{($k$-uniform) displacement tableau} if it
also satisfies the \define{displacement condition}:
\begin{quotation}
  whenever $t(x,y) = t(x',y')$, we have that
  $x-y\equiv x'-y' \pmod k$.\footnote{Or equivalently, if $(x,y)$ and
    $(x',y')$ contain the same symbol, then $(x,y)$ and $(x',y')$ must
    be separated by a lattice distance that is a multiple of $k$.}
\end{quotation}
This condition partitions $\lambda$ into $k$ regions.  To be precise,
we define the \define{$i$-th diagonal modulo $k$}, denoted by
$D_{i,k}$, to be the set of boxes
$\set{(x,y)\in\lambda \given x - y \equiv i \pmod k}$; then $\lambda$
is the disjoint union of $D_{i,k}$ for $i \in \Z/k\Z$, and the fiber
of each symbol of $t$ is contained in some $D_{i,k}$.  Co-opting
earlier terminology, we also call $k$ the \define{torsion} of $t$.

The $n$-th \define{anti-diagonal} $A_n$ is the set of all boxes
$(x,y)$ such that $x + y = n + 1$.  Define the \define{lower triangle
  of size $n$} to be $T_n \coloneq \bigcup_{i=1}^{n} A_i$.  For
example, \cref{fig:example-tableau} shows a lower-triangular
displacement tableau of size 6 and torsion 3.\footnote{We adopt the
  French notation, where the bottom-left box is $(1,1)$, the first
  coordinate increases to the right, and the second coordinate
  increases upwards.}
$D_{1,3}$ is blue, $A_6$ is red, and their intersection is purple.  Every
box $(x,y)$ here not colored red or purple is below (some box of) $A_6$.

\begin{figure}[htb]
  \centering
  \ytableausetup{centertableaux}
  \begin{ytableau}
    *(c3)11\\
    9&*(c2)10\\
    7&*(c1)8&*(c2)9\\
    *(c1)5&6&7&*(c3)8\\
    3&4&*(c1)5&6&*(c2)7\\
    1&*(c1)2&3&4&*(c1)5&*(c2)6
  \end{ytableau}
  \caption{A typical example of a lower-triangular displacement
    tableau of size 6 with torsion 3.}
  \label{fig:example-tableau}
\end{figure}

As explained in \cite[Section 3]{pflueger2017special}, $k$-uniform displacement tableaux on the rectangle
$[g-d+r]\times[r+1]$ with codomain $[g]$ give rise to divisors of degree $d$ and rank
at least $r$ on the uniform $k$-gonal chain of $g$ loops, as we now recall.  The
location $(x,y)$ of the symbol $a \in [g]$ in the tableau indicates where to place a
chip on the $a$-th loop.  Whenever $a$ does not appear in the tableau, the chip may
be placed arbitrarily on that loop, thereby allowing the locus a single degree of
freedom.  Otherwise, the $a$-th loop will have a chip at distance $m_a\cdot (x-y)$
counter-clockwise from its rightmost vertex (where the loops are arranged from left
to right, as in the bottom of \cref{fig:2}).  Finally place $d-g$ chips at the rightmost vertex of
the $g$-th loop.  The displacement condition guarantees that this is well-defined
when a symbol appears in the tableau more than once.

The tableau--divisor correspondence naturally extends to the folded chain of loops,
although the chips on the lower loops of $\widetilde\Gamma$ (as depicted in \cref{fig:2}) are measured
\emph{clockwise} from the \textit{leftmost} vertex, and the stack of $d-g$ chips are
placed at the \textit{leftmost} vertex of the $(2g-1)$-th loop.  As the genus of
the folded chain $\wti \Gamma$ is $2g-1$, the symbols should be taken from $[2g-1]$
and the domain should have shape $[2g-1-d+r]\times[r+1]$.  Since the parity of the $g$-th loop is
different  than the rest, the fiber of $g$ must be contained in $D_{i,2}$ for
some $i$.  By a slight abuse of terminology, we shall still refer to such tableaux as
``$k$-uniform.''

We wish to produce Prym divisors; these map down to $K_\Gamma$ and so must have
degree $2g-2$.  Hence, any tableau that yields Prym divisors under the correspondence
must be defined on the square domain $[r+1] \times [r+1]$.  Moreover, the
counter-clockwise distance of the chip on the $a$-th loop (for $a \in [g-1]$) must
equal the clockwise distance of the chip on the $(2g-a)$-th loop.  This motivates the
following \define{Prym condition}:
\begin{quotation}
  $t(x,y)=2g - t(x',y')$ only if $(x,y)$ and $(x',y')$ both lie in the
  same diagonal modulo $k$.
\end{quotation}

\begin{definition}
A tableau $t$ is \define{Prym of type $(g,r,k)$} if it
has shape $[r+1]\times [r+1]$ and codomain $[2g-1]$, it is $k$-uniform
(see above), and it satisfies the Prym condition.
\end{definition}

The two tropical Prym varieties arising from a folded chain of loops  are distinguished by the parity of the rank of the divisors that they classify \cite[Theorem 5.3.8]{len2019skeletons}. 
The parity, in turn, is determined by the placement of the chip on the $g$-th loop,
or equivalently, the position of the symbol $g$ in the tableau.  We denote by $P(t)$
the set of Prym divisors obtained from $t$ via the tableau--divisor correspondence
whose rank coincides with $r$ modulo $2$.  Explicitly, if $t^{-1}(g)$ is contained in
$D_{r,2}$, then $P(t)$ coincides with the set of divisors obtained from the
correspondence. If $t^{-1}(g)$ is contained in $D_{r+1,2}$, then $P(t)$ is
empty.\footnote{The special role that the symbol $g$ plays in determining $P(t)$ will
  cause minor headaches in \cref{sec:reflective}, but thereafter, we avoid the issue
  entirely by working (almost) exclusively with a different sort of tableau whose
  symbols only go up to $g-1$.}  Finally, if $t^{-1}(g)$ is empty, then $P(t)$ is a
proper subset of the divisors obtained from the correspondence. Either way, $P(t)$ is
a cell in the Prym variety.


\begin{remark}\label{rem:2}
  By \cite[Corollary 5.3.10]{len2019skeletons}, for fixed gonality $g$
  and torsion $k$, the Prym--Brill--Noether locus
  $V^r(\Gamma,\varphi)$ is the union of the subspaces $P(t)$, where $t$
  ranges over the Prym tableaux of type $(g,r,k)$. Moreover, it suffices to consider the tableaux for which the symbol $g$ is  in the ``correct'' diagonal modulo 2, namely, $D_{r,2}$.
\end{remark}

\section{Dimensions of Prym--Brill--Noether loci}

Our primary focus in this section is to prove \cref{thm:tropicalPBN}
by constructing Prym tableaux that---in a sense we shall make
precise---minimize the number of symbols used.  In
\cref{sec:reflective}, we describe a restricted class of Prym
tableaux, called reflective, that are easier to work with and
determine sets of divisors that are maximal with respect to
containment.  In \cref{sec:tropical-dim-proof}, we compute the largest
dimension of any cell determined by a reflective tableau (of fixed
type) and thereby compute the dimension of the Prym--Brill--Noether
locus.

\subsection{Reflective tableaux}\label{sec:reflective}

Fix a Prym tableau $t$ of type $(g,r,k)$.  We define the
\define{codimension} of $t$ to be number of integers $a \in [g-1]$ for
which either of the symbols $a$ or $2g-a$ appears in $t$.  By the
tableau--divisor correspondence, the codimension of $t$ coincides with
the codimension of the cell $P(t)$ relative to the Prym variety
(provided that $t^{-1}(g) \subset D_{r,2}$); indeed, there are at most
$g-1$ degrees of freedom---one for each loop $\ti\gamma_a$ with
$a \in [g-1]$---and the chip on the $a$-th loop is free just in case
neither $a$ nor $2g-a$ appears in the tableau.  Then the path to
proving \cref{thm:tropicalPBN} is clear:

\begin{quotation}
  To compute the codimension of $V^r(\Gamma,\varphi)$, it suffices to
  compute the minimal codimension of any Prym tableau of type
  $(g,r,k)$.
  \end{quotation}

To that end, it is beneficial to consider tableaux with a stronger symmetry than Prym tableaux. 
Given $\lambda \subset [r+1]\times[r+1]$, consider the map
$\rho\maps\lambda\to [r+1]\times[r+1]$ defined by $\rho(x,y)=(r+2-y, r+2-x)$; in
other words, $\rho$ picks out the box that is the reflection of $(x,y)$ across the
\define{main anti-diagonal}, $A_{r+1}$.  Fixing a map $t \maps \lambda \to [2g-1]$,
we say that a box $(x,y) \in \lambda$ is \define{reflective} (in $t$) provided that
$\rho(x,y) \in \lambda$ and $t(x,y) = 2g - t(\rho(x,y))$, i.e., the symbol in the box
is the dual of the symbol in its reflection.

\begin{definition}
  A displacement tableau $t$ is said to be \define{reflective} if
  every box of $t$ is reflective.
\end{definition}

Note that reflective tableaux defined on $[r+1]\times[r+1]$ are Prym.
Moreover, if $t$ is such a tableau, then each box along the main
anti-diagonal of $t$ must contain the symbol $g$, and $g$ appears
nowhere else.  In particular, $t^{-1}(g) \subset D_{r,2}$, so $P(t)$
is nonempty.

Our goal in the remainder of this section is to prove that, in our
search for Prym tableaux of minimal codimension, it suffices to
restrict our attention to the class of reflective tableaux.
\cref{prop:reflective} makes this precise, although we first need the
notion of tableaux containment that the next definition provides.

\begin{definition}\label{def:5}
  Given Prym tableaux $t$ and $s$ of type $(g,r,k)$, we say that $t$
  \define{dominates} $s$ if $g \in t(D_{i,2})$ implies that $g \in s(D_{i,2})$ and
  if, for any $a \neq g$ and $i \in \Z/k\Z$, $a \in t(D_{i,k})$ implies that either
  $a \in s(D_{i,k})$ or $2g-a \in s(D_{i,k})$.  If $t$ and $s$ each dominate the
  other, then we call them \define{equivalent}.
\end{definition}

It follows from the tableau--divisor correspondence that $t$ dominates $s$ only if
$P(t) \supset P(s)$.  Indeed, this containment holds whenever each chip that is fixed
in $P(t)$ is also fixed in $P(s)$ at the same coordinate.\footnote{The condition on
  the symbol $g$ in \cref{def:5} ensures that if $P(t)$ is empty, then so is $P(s)$.}
It follows that $t$ and $s$ are equivalent only if $P(t) = P(s)$.  If $s$ dominates
$t$, then $\codim(s) \leq \codim(t)$.  Therefore, for the purpose of computing the
dimension of $V^r(\Gamma,\varphi)$, we may restrict our attention to tableaux that
are maximal with respect to the partial order given by dominance.  The main result of
this section is the following.

\begin{proposition}\label{prop:reflective}
  Let $t$ be a Prym tableau such that $t^{-1}(g) \subset D_{r,2}$.  Then there exists
  a reflective tableau $s$ that dominates $t$.
\end{proposition}

The following definition from \cite{pflueger2017special} will be used
repeatedly during the proof. Given a partition $\lambda$ and a subset
$S \subset \Z/k\Z$, the \define{upward displacement of $\lambda$ by
  $S$}, denoted  $\disp(\lambda, S)$, is equal to $\lambda \cup L$,
where $L$ consists precisely of those boxes $(x,y) \nin \lambda$ such
that all of the following conditions hold:
\begin{itemize}
\item $(x-1,y) \in \lambda$ or $x=1$,
\item $(x,y-1) \in \lambda$ or $y=1$, and
\item $(x,y) \in D_{i,k}$ for some $i \in S$.
\end{itemize}
The boxes in $L$ are known as the \define{loose boxes of $\lambda$
  with respect to $S$}.  When $S = \Z/k\Z$, we use the shorthand
$\disp(\lambda)$ and note the following: if $\lambda$ is a partition,
then so is $\disp(\lambda)$; $L$ is nonempty; and every box in
$\N^2 \setminus \disp(\lambda)$ is above some box in $L$.  The
usefulness of this operation on partitions is made evident in the
following example, which outlines the subsequent proof of Proposition
\ref{prop:reflective}.  

\begin{example}\label{ex:1}
  Consider the first Prym tableau of type $(g,r,k)=(11,4,3)$ in the
  sequence illustrated in Fig.~\ref{fig:reflective}.  This tableau is
  far from being reflective, but at each step we make small changes so
  that the resulting tableau is closer to being reflective and
  dominates the preceding one.
  
  At each step, the boxes previously dealt with are colored blue.  We
  look at the symbols in the loose boxes with respect to the
  lower-left blue partition and choose the minimum $a$; we look at the
  symbols contained in the reflection of the loose boxes and choose
  the maximum $b$; then denote by $c$ the minimum of $a$ and $2g-b$.
  Now, wherever $c$ or $2g - c$ appears, color the corresponding box
  and its reflection red.  To produce the next tableau in the
  sequence, replace each symbol in the red-colored boxes with $c$ or
  $2g-c$ as appropriate.  The final tableau is reflective and
  dominates the initial tableau.

  \begin{figure}[htb]
    \input{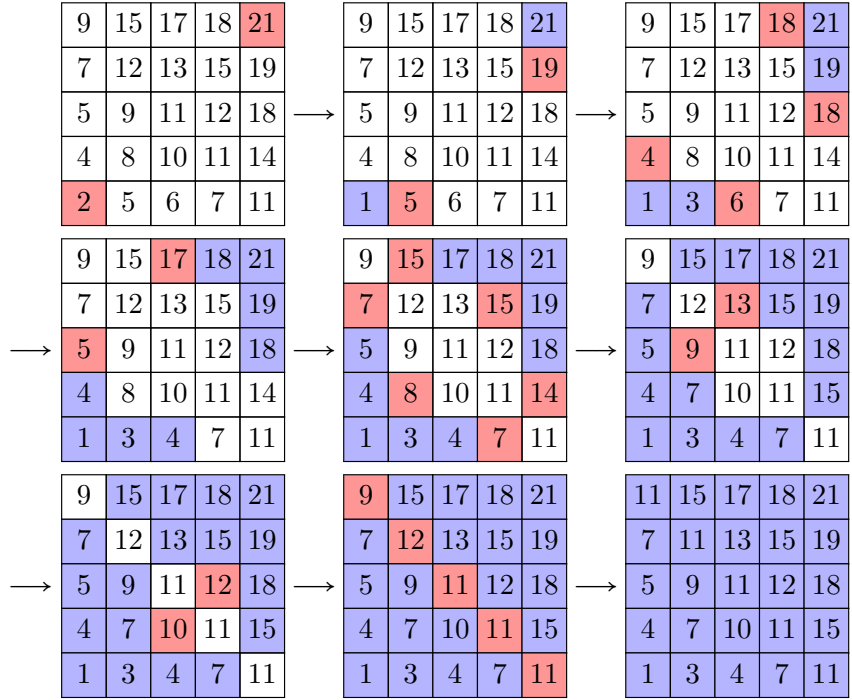}
    \caption{Replacing a non-reflective tableau with a dominant
      reflective one.}
    \label{fig:reflective}
  \end{figure}
\end{example}

The basic operation of the algorithm is to repeatedly \define{reflect}
symbols, i.e., given a box $(x,y)$, to insert the dual symbol,
$2g - t(x,y)$, into the reflection, $\rho(x,y)$.  The following
lemma ensures that the result is still a Prym tableau, granted that
the tableau condition holds; then the proof of Proposition
\ref{prop:reflective} will make the rest of the algorithm precise.

\begin{lemma}\label{lem:5}
  Given a Prym tableau $t$, fix a box $(x,y)$.  Then the map $s$
  obtained by defining
  \begin{equation*}
    s(\omega) =    
    \begin{cases}
      2g - t(x,y) &\text{for } \omega = \rho(x,y) \\
      t(\omega) &\text{otherwise}
    \end{cases}
  \end{equation*}
  satisfies the displacement and Prym conditions.
\end{lemma}

\begin{proof}
  The only box at which either of the conditions might fail is at
  $\rho(x,y)$.  However, taking the difference of the coordinates of
  $\rho(x,y) = (r+2-y,r+2-x)$, we find that $\rho(x,y) \in D_{x-y,k}$.
  The Prym condition is immediately satisfied, and it is not hard to
  see that, since any other box containing the symbol $2g - t(x,y)$
  would need to be in $D_{x-y,k}$, the displacement condition is also
  satisfied.
\end{proof}

\begin{proof}[Proof of \cref{prop:reflective}]
  Let  $\lambda \coloneq [r+1]\times[r+1]$ be the domain of $t$, and let
  $s_0 \coloneq t$. We describe an algorithm which at each step, given a Prym tableau
  $s_i$, produces a Prym tableau $s_{i+1}$ that dominates $s_i$ and is reflective on
  a larger subset of $\lambda$.  After a finite number of steps, we obtain a Prym
  tableau $s_m$ that is reflective away from the main anti-diagonal and that
  dominates $t$ by transitivity.  In the final step, the symbols along the main
  anti-diagonal of $s_m$ are replaced with $g$ to obtain a reflective tableau $s$.

  \textit{Induction hypotheses.}  Suppose that after the $i$-th step we have a Prym
  tableau $s_i$ that dominates $s_{i-1}$.  Furthermore, suppose that we have an
  integer $0\leq n_i\leq g-1$ and a subset $\kappa_i \subset \lambda$ (where, by
  convention, $n_0=0$ and $\kappa_0=\emptyset$) such that
  \begin{itemize}
  \item $\omega \in \kappa_i$ just if $\omega \in T_r$ and $s_i(\omega) \leq n_i$,
  \item $\omega \in \rho(\kappa_i)$ just if $\omega \in \rho(T_r)$ and
    $s_i(\omega) \geq 2g - n_i$,
  \item $s_i\restrict{\kappa_i \cup \rho(\kappa_i)}$ is reflective.\footnote{In
      \cref{ex:1}, $\kappa_i \cup \rho(\kappa_i)$ is represented by the blue boxes.}
  \end{itemize}

  If $\kappa_i = T_r$, then $i = m$ and we are ready to perform the final step.
  Otherwise, note that $\kappa_i$ must be a partition by the tableau condition.
  Therefore, let $L_i$ be the set of loose boxes of $\kappa_i$ that lie below the
  main anti-diagonal, and observe that $L_i$ is nonempty.

  \textit{Definition of $s_{i+1}$.}  Consider the minimal positive integer $n_{i+1}$
  among the set of symbols $s_i(L_i) \cup (2g - s_i(\rho(L_i)))$.  We claim that
  $n_{i+1}$ exists and is at most $g-1$.  Indeed, given any $\omega \in L_i$, if
  $s_i(\omega) \leq g-1$, then we are done.  Otherwise, $s_i(\omega) \geq g$; since
  $\omega$ lies below its reflection $\rho(\omega)$, it follows that
  $s_i(\rho(\omega)) \geq g+1$; moreover, $s_i$ is Prym, so it must be the case that
  $s_i(\rho(\omega)) \leq 2g-1$; from both of these inequalities, we find that
  $1 \leq 2g-s_i(\rho(\omega)) \leq g-1$.  In either case, the claim holds.  We also
  know from the assumptions on $\kappa_i$ and its reflection that
  $n_{i+1} \geq n_i+1$.
  
  We now define
  \begin{equation*}
    N_{i+1}\coloneq \set{ \omega\in L_i \given s_i(\omega) = n_i \text{ or }
      s_i(\rho(\omega)) = 2g - n_i}.
  \end{equation*}
  In going from $s_i$ to $s_{i+1}$, we modify only the boxes in
  $N_{i+1} \cup \rho(N_{i+1})$.\footnote{In \cref{ex:1}, $N_{i+1} \cup \rho(N_{i+1})$
    is represented by the red boxes.}  In particular, the symbol $n_{i+1}$ is placed
  in each box of $N_{i+1}$ (that did not already contain it) while $2g - n_{i+1}$ is
  placed in $\rho(N_{i+1})$.  Thus, we define $s_{i+1}$ by
  \begin{equation*}
    \label{eq:10}
    s_{i+1}(\omega) = 
    \begin{cases}
      n_{i+1} &
      \text{for } \omega \in N_{i+1} \\
      2g - n_{i+1} &
      \text{for } \omega \in \rho(N_{i+1}) \\
      s_i(\omega) & \text{otherwise }
    \end{cases}.
  \end{equation*}
  
  \textit{Proof that $s_{i+1}$ is a Prym tableau.}  
  By applying  \cref{lem:5} every time a symbol is replaced, we know that $s_{i+1}$
  satisfies the displacement and Prym conditions.  
  It remains to show that it satisfies the tableau condition at the modified
  boxes.  We shall consider only the case where $\omega\in N_{i+1}$; the case where
  $\omega\in \rho(N_{i+1})$ follows in a similar way.  Observe first that every box
  below $\omega$ contains a symbol that is smaller than $n_{i+1}$. Indeed, since
  $\omega\in L_i$, it follows that every box below $\omega$ lies in $\kappa_i$; since
  the maximum value of a symbol in $\kappa_i$ is $n_i$ and we know that
  $n_i<n_{i+1}$, every box below $\omega$ contains a symbol that is strictly smaller
  than $s_{i+1}(\omega)$.

  We also claim that the boxes immediately above $\omega$ contain symbols that are
  greater than $n_{i+1}$.  Writing $\omega=(x,y)$, observe that
  $(x+1,y) \nin L_i \cup \rho(L_i)$, so $s_{i+1}(x+1,y) = s_i(x+1,y)$.  By the
  tableau condition on $s_i$, we have that $s_i(x+1,y) > s_i(x,y)$.  Finally, $(x,y)$
  is in $L_i$ and $n_{i+1}$ was chosen to be minimal among the symbols of $L_i$ (in
  particular), so we get that $s_i(x,y) \geq n_{i+1} = s_{i+1}(x,y)$.  Chaining these
  inequalities together yields $s_{i+1}(x+1,y) > s_{i+1}(x,y)$; the same argument
  works for $(x,y+1)$.  Hence, $s_{i+1}$ satisfies the tableau condition and so is a
  Prym tableau.
  
  
  \textit{Proof that $s_{i+1}$ dominates $s_i$.}  Let $\omega \in \lambda$.  If
  $s_{i+1}(\omega) \nin \set{n_{i+1}, 2g-n_{i+1}}$, then
  $s_{i+1}(\omega) = s_i(\omega)$.  This implies that for all symbols besides
  $n_{i+1}$ and $2g-n_{i+1}$ (including $g$), the conditions of \cref{def:5} are
  satisfied.  If $s_{i+1}(\omega) \in \set{n_{i+1}, 2g-n_{i+1}}$, then either
  $s_{i+1}(\omega) = s_i(\omega)$ or $s_{i+1}(\omega) = 2g - s_i(\rho(\omega))$.  In
  the first case, we are done as above; in the second case, the desired condition
  still holds on account of the fact that the dual of the symbol $s_{i+1}(\omega)$
  appears in $s_i$ and, in particular, is contained in $\rho(\omega)$, which occupies
  the same diagonal modulo $k$ as $\omega$.  Hence, $s_{i+1}$ dominates $s_i$.

  \textit{Proof that $s_{i+1}$ satisfies the induction hypotheses.}  Define
  $\kappa_{i+1}$ to be $\kappa_i \cup N_{i+1}$.  Then every box in $T_r$ that
  contains a symbol at most $n_i$ is in $\kappa_i$, while any box containing
  $n_{i+1}$ is in $N_{i+1}$.  Using the definition of loose boxes and the fact that
  $n_{i+1}$ minimizes the symbols in $L_i$, we find that no symbol strictly between
  $n_i$ and $n_{i+1}$ appears in $s_{i+1}$.  Moreover, any box in $T_r$ is above some
  box of $L_i$, so the tableau condition precludes $n_{i+1}$ from appearing in
  $T_r \setminus L_i$; if $n_{i+1}$ appears in $L_i$, then it appears in $N_{i+1}$ by
  definition.  From these observations, we find that $\kappa_{i+1}$ contains
  precisely those boxes of $T_r$ with symbols at most $n_i$.  A similar argument
  shows that $\rho(\kappa_{i+1})$ contains precisely those boxes of $\rho(T_r)$ with
  symbols at least $2g-n_i$.

  Finally, the restriction of $s_{i+1}$ to $\kappa_{i+1} \cup \rho(\kappa_{i+1})$ is
  reflective.  Indeed, it is reflective on $\kappa_i \cup \rho(\kappa_i)$ because
  $s_i$ is, and $s_{i+1}$ and $s_i$ agree on that subset.  Moreover, $s_{i+1}$ is
  reflective on $N_{i+1} \cup \rho(N_{i+1})$ by construction: every symbol in this
  subset is the dual of the symbol in its reflection.  Therefore, all the inductive
  hypotheses are satisfied.

  \textit{Final step.}  Since $\kappa_{i+1}$ strictly contains $\kappa_{i}$, after a
  finite number of steps $m$, we have that $\kappa_m = T_r$.  In other words, $s_m$
  is reflective everywhere but (possibly) the main anti-diagonal, $A_{r+1}$.  Then we
  replace the symbols in all of the boxes in $A_{r+1}$ with the symbol $g$; the
  resulting tableau $s$ is reflective and dominates $s_m$, completing the
  proof.\footnote{It may not be obvious at first glance why the statement of
    \cref{prop:reflective} requires that $t^{-1}(g) \subset D_{r,2}$.  If
    $g \in t(D_{r+1,2})$, then the algorithm described in this proof still produces a
    reflective tableau $s$.  However, the final step forces $g \in s(D_{r,2})$, so
    $s$ would fail to dominate $t$ in this case.}
\end{proof}

A reflective tableau is uniquely determined by its restriction to $T_r$, so we may as
well only consider this subset.

\begin{definition}
  A \define{staircase Prym tableau of type $(g,r,k)$} is a $k$-uniform displacement
  tableau $t \maps T_r \to [g-1]$.
\end{definition}
We extend all definitions regarding Prym tableaux to staircase Prym tableaux in the
natural way; for instance, denoting by $\hat{t}$ the reflective tableau which extends
a given staircase Prym tableau $t$, we define $P(t)$ to be just $P(\hat{t})$.
Certain other definitions become more intuitive: the codimension of $t$ is just the
number of distinct symbols appearing in $t$, and $t$ dominates another staircase Prym
tableau $s$ just in case $t(D_{i,k}) \subset s(D_{i,k})$.

\subsection{Proof of
  \cref{thm:tropicalPBN}}\label{sec:tropical-dim-proof}

Throughout this section, $\varphi\maps\wti\Gamma\to\Gamma$ will
represent a folded chain of loops of genus $g$, where the edge lengths
of $\Gamma$ are either generic or the torsion of each loop is $k$. For
the sake of brevity, we will refer to the folded chain of loops and
its corresponding Prym tableaux in the former case as \define{generic}
and in the latter as \define{$k$-gonal}.

The dimension of $V^r(\Gamma,\varphi)$ is known in the generic case
and when $k$ is even; see \cite[Theorem~6.1.4,
Corollary~6.2.2]{len2019skeletons}.  When $k$ is odd,
\cite[Remark~6.2.3]{len2019skeletons} provides an upper and a lower
bound for the dimension. In this section we show that the dimension of
$V^r(\Gamma, \varphi)$ in fact coincides with the lower bound.  We
restate the precise result here.

\tropicalPBN*



To prove the theorem, we need to compute the minimal codimension of $P(t)$ over all
Prym tableaux $t$ of type $(g,r,k)$ such that $t^{-1}(g) \subset D_{r,2}$ (see
\cref{rem:2}).  By \cref{prop:reflective}, it suffices to consider staircase Prym
tableaux: given any Prym tableau (with the correct $g$-fiber), we apply the
reflection algorithm to obtain a dominating reflective Prym tableau. Per the
discussion at the end of \cref{sec:reflective}, it then suffices to consider the
staircase Prym tableau that constitutes its restriction to $T_r$.

The expression $\binom{r+1}{2}$ in the second case in \cref{eq:4} counts the number
of boxes in $T_r$.  In this subset, the lattice distance between any two boxes is at
most $2r-2 \leq 2l - 2 < k$, so each must contain a unique symbol; it follows that
the number of symbols in any such tableau is precisely $\binom{r+1}{2}$.  The same
reasoning explains the presence of the $\binom{l+1}{2}$ term in the first case: it
counts the number of symbols in $T_l$, which are all necessarily unique.  Any repeats
occur above $T_l$.  In fact, we claim that a tableau of minimal codimension contains
precisely $l$ new symbols on each subsequent anti-diagonal, of which there are $r-l$;
this accounts for the $l(r-l)$ term.  Precisely, we say that a set of symbols
$S \subset t(A_n)$ is \define{new} if $S \cap t(T_{n-1})$ is empty.  If our claim is
true, then the tableau depicted in \cref{fig:example-tableau}, which is a staircase
Prym tableau of type $(12,6,3)$,\footnote{In fact, it is staircase Prym of type
  $(g,6,3)$ for any $g \geq 12$.} has minimal codimension.

\begin{proposition}
  \label{prop:1}
  Given a staircase Prym tableau $t$ of type $(g,r,k)$, there exist at
  least $l$ new symbols in $A_n$ for each $n \geq l + 1$.
\end{proposition}

\noindent The following lemma establishes a restriction on symbols
which will go most of the way toward proving \cref{prop:1}, from which
the proof of \cref{thm:tropicalPBN} quickly follows.

\begin{lemma}
  \label{lem:2}
  Let $t$ be a staircase Prym tableau of type $(g,r,k)$, and fix
  $n\leq r$.  For any boxes $(x,y) \in D_{i,k}$ and $(x',y') \in D_{i+1,k}$
  that lie below $A_n$, there exists a box
  $\omega \in A_n\cap(D_{i,k}\cup D_{i+1,k})$ such that $t(\omega)$ is
  greater than both $t(x,y)$ and $t(x',y')$.
\end{lemma}

\begin{proof}
  Let $a= t(x,y)$ and $b= t(x',y')$.  Since $a$ and $b$ lie in different diagonals
  modulo $k$, we know that $a \neq b$.  We will assume that $a<b$; the proof follows
  in the same way when the converse inequality holds.  We want to show that there is
  a box $\omega \coloneq (\omega_1,\omega_2)$ in $A_n\cap(D_{i,k}\cup D_{i+1,k})$
  that lies above $(x',y')$, since this would force $t(\omega) > b$.

  Indeed, define $\delta = n+1-x'-y'$.  We know that $x'+y' \leq n$
  because $(x',y')$ sits below $A_n$, so $\delta \geq 1$.  If $\delta$
  is even, then we define
  \begin{equation*}
    \label{eq:6}
    \omega \coloneq
    \left(x'+\frac{\delta}{2},y'+\frac{\delta}{2}\right).
  \end{equation*}
  Note that $\omega_1$ and $\omega_2$ are both positive integers,
  $\omega_1+\omega_2=n+1$, and
  $\omega_1-\omega_2 = x'-y' \equiv i+1 \pmod k$; moreover,
  $\omega$ sits above $(x',y')$, as desired.

  Suppose instead that $\delta$ is odd; then define
  \begin{equation*}
    \label{eq:7}
    \omega \coloneq
    \left( x'+\frac{\delta-1}{2} ,
           y'+\frac{\delta+1}{2} \right).
  \end{equation*}
  The desired properties once again hold (although in this case,
  $\omega \in D_{i,k}$).
\end{proof}

\begin{proof}[Proof of \cref{prop:1}]
  Given $n$ such that $l+1 \leq n \leq r$, we note first that
  $T_{n-1} \cap D_{i,k}$ is nonempty.  Indeed, we may write
  $i \in \set{-l+1,\ldots,l-1}$.  If $i \geq 0$, we have that
  $(1+i,1) \in T_{n-1} \cap D_{i,k}$; if $i < 0$, then
  $(1,1-i) \in T_{n-1} \cap D_{i,k}$.

  For each $i$, choose $\omega_i \in T_{n-1} \cap D_{i,k}$ such that
  $t(\omega_i)$ is maximal among $t(T_{n-1} \cap D_{i,k})$.  Then apply
  \cref{lem:2} to each pair $\set{\omega_i, \omega_{i+1}}$ to obtain a
  box $\eta_i \in A_n \cap (D_{i,k} \cup D_{i+1,k})$ such that
  $t(\eta_i) > t(\omega_i)$ and $t(\eta_i) > t(\omega_{i+1})$.  Hence,
  $t(\eta_i) > t(\omega)$ for every box
  $\omega \in T_{n-1} \cap (D_{i,k} \cup D_{i+1,k})$ and so is new in $A_n$.

  Therefore, for each pair $\set{i,i+1} \subset \Z/k\Z$, the set
  $A_n \cap (D_{i,k}\cup D_{i+1,k})$ contains at least one new symbol, which we shall
  denote by $b_i$.  Note that if $\set{i,i+1}$ and $\set{j,j+1}$ are disjoint, then
  their respective symbols $b_i$ and $b_j$ must lie in different diagonals modulo
  $k$, and so must be distinct.  Thus, the minimum number of new symbols in $A_n$
  coincides with the minimum number of elements we can choose from $\Z/k\Z$ such that
  we have at least one element in each pair $\set{i,i+1}$.  Suppose for the sake of
  contradiction that we could achieve this with $l-1$ elements.  Each is a member of
  two pairs, so we cover at most $2(l-1)<k$ pairs.  This is insufficient, as there
  are $k$ pairs, so the minimum size of such a set is $l$.
\end{proof}

\begin{proof}[Proof of \cref{thm:tropicalPBN}]
  We have already proved the case where $l > r$, so assume otherwise.
  From \cref{prop:1} and our earlier remarks, we get that $T_r$
  contains at least $\binom{l+1}{2} + l(r-l)$ distinct symbols.
  Hence, $\codim V^r(\Gamma,\varphi)$ is bounded below by this
  quantity.  Meanwhile, \cite[Corollary~6.2.2, Remark~6.2.3]{len2019skeletons}
   implies that it is also
  an upper bound, so we are done.
\end{proof}

\subsection{Relation to algebraic
  geometry}\label{sec:algebraic-dim-proof}
We are now in a position to prove \cref{cor:algebraicPBN}, restated
below.  

\algebraicPBN*

\begin{proof}
  Having established Theorem \ref{thm:tropicalPBN}, the proof of the Corollary is
  almost identical to the proof of \cite[Theorem~B]{len2019skeletons} and similar to
  analogous results from \cite{CDPR, JR, pflueger2017special}.  We illustrate the
  general idea, and leave the details to the reader. First, due to our assumption
  that the characteristic of the residue field is prime to both $2$ and $k$, we may
  lift the folded chain of loops $\varphi:\wti{\Gamma}\to\Gamma$ to a $k$-gonal
  unramified double cover $f \maps \wti{X} \to X$
  \cite[Lemma~7.0.1]{len2019skeletons}. By Baker's specialization lemma
  \cite[Corollary 2.11]{Baker_specialization}, the tropicalization of $V^r(X,f)$ (if
  non-empty) lies within $V^r(\Gamma,\varphi)$. By Gubler's Bieri--Groves Theorem
  \cite[Theorem 6.9]{Gubler_trop&nonArch}, dimensions are preserved under
  tropicalization, so the codimension of $V^r(X,f)$ inside the Prym variety is
  bounded from below by $n(r,k)$. A standard upper semicontinuity argument shows that
  $n(r,k)$ is, in fact, an upper bound on the codimension for a non-empty open  set in
  the $k$-gonal locus of $\cR_g$, as claimed.
\end{proof}
  
%
%

Note that this bound is not necessarily strict. For instance, if $g\leq 2k-2$, then
the general curve is $k$-gonal. In this case, the codimension of the
Prym--Brill--Noether locus of a general curve is $\binom{r+1}{2}$
\cite{Welters_Prym}, which is stronger than the bound provided in Corollary
\ref{cor:algebraicPBN}.  However, we believe that our bound is strict when $g$ is
sufficiently high.

\begin{conjecture}\label{conj:dimension}
  Suppose that $g\gg n(r,k)$, and let $f\maps\wti{C}\to C$ be a
  generic Prym curve. Then
  \begin{equation*}
    \label{eq:2}
    \dim V^r(C,f) = g-1-n(r,k).
  \end{equation*}
\end{conjecture}

\section{Tropological properties}\label{sec:trop-results}

As before, fix a folded chain of loops
$\varphi\maps\wti\Gamma\to\Gamma$ of genus $g$ and gonality $k$.  In
this section, we prove that the Prym--Brill--Noether locus
$V^r(\Gamma,\varphi)$ is pure-dimensional (\cref{thm:pure-dim}) and
connected in codimension 1 when the dimension is greater than zero
(\cref{thm:path}).  In \cref{sec:an-interl}, we develop the notions of
strips and non-repeating tableaux, which will also be necessary for
computing the Betti number of $V^r(\Gamma,\varphi)$ in
\cref{sec:counting}.\footnote{We use the term \define{Betti number} for
  the genus of the Prym Brill--Noether locus to distinguish it from
  the genus of our underlying graphs.}  The proof of
pure-dimensionality then comes as an easy corollary of \cref{prop:14}.
We tackle connectedness in \cref{sec:connect}.

\subsection{Strips and non-repeating tableaux}\label{sec:an-interl}

We focus our attention on Prym tableaux of minimal codimension.  Since
\cref{prop:reflective} implies that any such tableau is equivalent to
a reflective tableau and hence a staircase Prym tableau, it
suffices to consider this restricted type.  To simplify our
terminology, we shall say that a tableau is \define{minimal} if it is
staircase Prym of minimal codimension.

In the generic case (which, by a slight abuse of terminology, we take
to include both the case of generic edge lengths and the non-generic
case with $l\geq r$), minimal tableaux are relatively easy to
classify, since they are precisely the standard Young tableaux on
$T_r$.  The cases of even and odd torsion elude such a concise
description; nonetheless, as we will presently make precise, there are
subsets of $T_r$ that we call strips on which minimal tableaux are
determined up to equivalence.

\begin{definition}
  A subset $\mu \subset T_r$ is a \define{strip} if $T_l \subset \mu$
  and there exists a unique box in $\mu \cap A_n$ for each
  $n \in \set{l,l+1,\ldots,r}$ called the \define{$n$-th leftmost box}
  that satisfies the following properties:
  \begin{itemize}
  \item $(1,l)$ is the $l$-th leftmost box, 
  \item if $(x,y)$ is the $n$-th leftmost box, then the
    $(n+1)$-th leftmost box is $(x,y+1)$ or $(x+1,y)$, and
  \item if $(x,y)$ is the $n$-th leftmost box, then the boxes of
    $\mu \cap A_n$ are precisely those of the form $(x+i,y-i)$ for
    each $i \in \set{0,1,\ldots,l-1}$.
  \end{itemize}
  If $(x,y)$ is the $n$-th leftmost box, then we call $(x+l-1,y-l+1)$
  the \define{$n$-th rightmost box}.  We call $r$ and $l$ the
  \define{length} and \define{width} of $\mu$, respectively.
\end{definition}
  
Note that $\mu \cap A_n$ contains precisely $\min \set{n,\,l}$ boxes,
any two of which are separated by lattice distance at most $2l-2$.
This implies that any $k$-uniform tableau defined on $T_r$ must be
injective on each $\mu \cap A_n$.  Moreover, since we designate
$(1,l)$ as the $l$-th leftmost box and choose each subsequent leftmost
box out of two possibilities, it follows that $\mu$ may take on any of
$2^{r-l}$ distinct shapes.

$T_r\setminus\mu$ consists of two (possibly empty) contiguous
components, which we shall call the \define{left} and \define{right},
respectively.  In particular, the left component of $T_r\setminus\mu$
(if it exists) is the one that contains the box $(1,r)$.  We refer to
the strip whose right component is empty as the \define{horizontal
  strip} and denote it by $\mu_0$.

We now introduce a subclass of maps $T_r \to [g-1]$ that will play a
key role for the rest of the paper.  The even and odd cases differ; in
what follows, let $\epsilon$ be $0$ if $k$ is even and $1$ if $k$ is odd.

\begin{definition}\label{def:2}
  Given a strip $\mu$ and a map $t\maps T_r \to [g-1]$ such that
  $t\restrict{\mu}$ satisfies the tableau and displacement conditions,
  we say that $t$ is \define{non-repeating in $\mu$} if
  \begin{enumerate}
  \item\label{def:2a} $t(x,y) = t(x+l-\epsilon,y-l)$ for each $(x,y)$
    in the left component of $T_r\setminus\mu$,
  \item\label{def:2b} $t(x,y) = t(x-l,y+l-\epsilon)$ for each $(x,y)$
    in the right component of $T_r\setminus\mu$, and
  \item\label{def:2c} writing the $n$-th leftmost box as $(x,y)$, if
    $(x+1,y)$ is the $(n+1)$-leftmost box, then
    $t(x,y) < t(x+l-\epsilon,y-l+1)$; otherwise,
    $t(x+l-1,y-l+1) < t(x,y+1-\epsilon)$.
  \end{enumerate}
  We refer to \ref{def:2a} and \ref{def:2b} as the \define{left} and
  \define{right} \define{repeating conditions} respectively and to
  \ref{def:2c} as the \define{gluing condition}.
\end{definition}

See \cref{fig:strip-example} for an example.  It is straightforward to
check that these conditions are symmetrical with respect to
transposing the first and second coordinates.\footnote{To be precise,
  we transpose by switching the coordinates of each box, replacing $x$
  with $y$ and vice versa, and replacing ``left'' with ``right'' and
  vice versa.}  This fact will simplify the proofs of several
properties of non-repeating maps.

As a small convenience, we shall use the cardinal directions to refer
to boxes relative to a given box, with east and north corresponding to
increasing first and second coordinates, respectively.  So for
example, the north neighbor of $(x,y)$ is $(x,y+1)$, and the box three
steps west of $(x,y)$ is $(x-3,y)$.  We shall also call the north and
east neighbors the \define{upper neighbors} and the south and west
neighbors the \define{lower neighbors}.

\begin{figure}[htb]
  \centering
  \begin{ytableau}
    20      \\
    18      & 21      \\
    15      & 17      & *(c1)23 \\
    11      & 14      & *(c1)20 & *(c1)24 \\
    *(c1)10 & *(c1)13 & *(c1)18 & *(c1)21 & *(c1)22 \\
    *(c1)8  & *(c1)12 & *(c1)15 & *(c1)17 & *(c1)19 & 20 \\
    *(c1)6  & *(c1)9  & *(c1)11 & *(c1)14 & *(c1)16 & 18 & 21 \\
    *(c1)3  & *(c1)5  & *(c1)7  & 8 & 12  & 15 & 17 & 19 \\
    *(c1)1  & *(c1)2  & *(c1)4  & 6 & 9   & 11 & 14 & 16 & 18 \\
  \end{ytableau}
  \caption{A minimal tableau of size 9 and torsion 5 that is
    non-repeating on the strip depicted in blue.}
  \label{fig:strip-example}
\end{figure}

\begin{proposition}\label{prop:18}
  Given a strip $\mu$, if $t$ is non-repeating in $\mu$, then $t$ is a
  minimal tableau.
\end{proposition}

\begin{proof}
  We first show that $t$ is staircase Prym.  The displacement
  condition holds in $\mu$ by definition and in $T_r \setminus \mu$ by
  the repeating conditions, since each symbol in $T_r \setminus \mu$
  is copied from a box that is distance $k$ away.

  Recall from the definition that $t\restrict\mu$ satisfies the
  tableau condition.  We need to check that $t$ satisfies the tableau
  condition.  We observe first that $t(1,1)$ is smaller than both
  $t(2,1)$ and $t(1,2)$.  This follows in the case that $k = 2$
  because both $(2,1)$ and $(1,2)$ contain the same symbol and one of
  the two is in $\mu$ along with $(1,1)$; in the case that $k > 2$,
  all three boxes are in $\mu$.  Suppose for the sake of induction
  that the tableau condition holds for all boxes in $T_{n-1}$ (and in
  particular, every box in $A_{n-1}$ contains a symbol smaller than
  the symbols of its upper neighbors).  Let $(x,y)$ be a box in $A_n$.
  If $n = r$, we are done; otherwise, it suffices to show that
  $t(x,y) < t(x+1,y)$ and $t(x,y) < t(x,y+1)$.

  By transposing the coordinates if necessary, we may assume that the
  $(n+1)$-th leftmost box is east of the $n$-th leftmost box.  Suppose
  first that $(x,y)$ is in $\mu$.  If it is not the $n$-th leftmost
  box, both of the desired inequalities follow from the fact that
  $(x+1,y)$ and $(x,y+1)$ are both also in $\mu$.  Otherwise, its east
  neighbor is in $\mu$ while its north neighbor is in the left
  component of $T_r \setminus \mu$.  We use the left repeating
  condition followed by the gluing condition to obtain the desired
  inequality:
  \begin{align*}
    t(x,y+1) = t(x+l-\epsilon,y-l+1) > t(x,y).
  \end{align*}

  Now suppose that $(x,y)$ is in the left component.  If $k$ is odd,
  then the symbols in $(x,y)$ and its upper neighbors are copied from
  the respective symbols in $(x+l-1,y-l)$ and \textit{its} upper
  neighbors.  Since $(x+l-1,y-l)$ is in $A_{n-1}$, it satisfies the
  tableau condition by the induction hypothesis.  If $k$ is even, we
  observe via repeated application of the left repeating condition
  that there is some box $(x',y')$ in $\mu \cap A_n$ such that the
  symbols in $(x,y)$ and its upper neighbors are copied from the
  respective symbols in $(x',y')$ and \textit{its} upper neighbors.
  We checked that the desired inequalities hold for every box in
  $\mu \cap A_n$, so they hold at $(x,y)$ as well.  Analogous
  arguments hold in both the odd and even cases when $(x,y)$ is in the
  right component.

  By induction, $t$ satisfies the tableau condition.  Thus, $t$ is
  staircase Prym.  It remains to show that $t$ has minimal
  codimension.  Indeed, observe that $\mu$ consists of
  $n(r,k) = \binom{l + 1}{2} + l(r-l)$ boxes, so $t\restrict{\mu}$
  contains at most that many distinct symbols.  Every symbol of $t$ in
  $T_r \setminus \mu$ is repeated from within $\mu$, so $t$ as a whole
  contains at most $n(r,k)$ symbols.  By \cref{thm:tropicalPBN}, $t$
  is minimal.
\end{proof}
  
\begin{corollary}\label{cor:1}
  If $t$ is non-repeating in $\mu$, then $t\restrict{\mu}$ is
  injective.
\end{corollary}

\begin{proof}
  By applying \cref{thm:tropicalPBN}, we know that $t$ cannot contain
  fewer than $n(r,k)$ symbols.  Every symbol of $t$ appears in $\mu$,
  and $\mu$ consists of precisely $n(r,k)$ boxes.  Hence, each of
  those boxes must contain a distinct symbol.
\end{proof}

\begin{lemma}\label{lem:19}
  Fix a strip $\mu$ and $i \in \Z/k\Z$.  Then for any map $t$
  non-repeating in $\mu$ and any boxes
  $\omega \in \mu \cap A_m \cap D_{i,k}$ and
  $\omega' \in \mu \cap A_n \cap D_{i,k}$ with $m < n$, it must be the
  case that $t(\omega) < t(\omega')$.
\end{lemma}
\begin{proof}
  The statement is true for $n \leq l$ by the tableau condition on
  $t\restrict{\mu}$ since $D_{i,k} \cap T_l$ is contained in a single
  diagonal.  We proceed by induction on $n$.  Suppose that the
  statement is true in $T_n$, and let $\omega'$ be a box in
  $\mu \cap A_{n+1} \cap D_{i,k}$.  Since $\mu \cap A_m \cap D_{i,k}$ contains
  at most one box, it suffices to show that $t(\omega) < t(\omega')$
  for $\omega \in \mu \cap A_m \cap D_{i,k}$ where $m \leq n$ is the
  maximum index such that $\mu \cap A_m \cap D_{i,k}$ is nonempty.

  Suppose that $k$ is odd.  Let $(x,y)$ be the $n$-th leftmost box,
  and assume without loss of generality that $(x+1,y)$ is the
  $(n+1)$-th leftmost box.  If $\omega'$ is the $(n+1)$-th rightmost
  box, $(x+l,y-l+1)$, then $m = n$ and $\omega = (x,y)$.  Then
  $t(\omega) < t(\omega')$ by the gluing condition.  If $\omega'$ is
  any other box $(x',y')$ in $\mu \cap A_{n+1} \cap D_{i,k}$, it is not
  hard to see that $\mu \cap A_n \cap D_{i,k}$ is empty and
  $\mu \cap A_{n-1} \cap D_{i,k}$ contains precisely one box; namely,
  $(x'-1,y'-1)$.  Then $\omega = (x'-1,y'-1)$, and the tableau
  condition on $t\restrict{\mu}$ implies that
  $t(\omega) < t(\omega')$.

  The case where $k$ is even follows in a similar way; we omit the
  details here.  
\end{proof}

Beginning with the following proposition, we start to see that the odd
and even cases are fundamentally different.  In particular, even
tableaux that are non-repeating on some strip are in fact
non-repeating on every strip; shortly, we will restrict our attention
to the horizontal strip whenever we talk about the even case.

\begin{proposition}\label{prop:2}
  Let $t$ and $s$ be tableaux that are non-repeating in $\mu$ and
  $\nu$ respectively.  For $k$ odd, $t$ and $s$ are equivalent if and only if
  $\mu=\nu$ and $t\restrict\mu = s\restrict\nu$.  For $k$ even, if
  $\mu = \nu$, then $t$ and $s$ are equivalent if and only if
  $t\restrict\mu = s\restrict\nu$.
\end{proposition}

\begin{proof}
  The converse of each statement trivially follows from \cref{def:2}.  For the
  forward direction, suppose that $t$ and $s$ are equivalent, and assume for the time
  being that $\mu = \nu$.  Because $t(D_{i,k}) = s(D_{i,k})$ for each $i$, the total
  ordering on the boxes of $\mu \cap D_{i,k}$ given by \cref{lem:19} forces
  $t\restrict{\mu} = s\restrict{\nu}$.  This proves the statement in the even case.
  For the odd case, suppose for the sake of contradiction that $\mu \neq \nu$, and
  let $n$ be the smallest index such that $\mu \cap A_{n+1} \neq \nu \cap A_{n+1}$
  (and note in particular that $n \geq l$).  Applying the argument above to the
  restricted domain $T_n$, we have that
  $t\restrict{\mu \cap T_n} = s\restrict{\nu \cap T_n}$.  Then the gluing condition
  on $\mu$ and $\nu$ forces the $(n+1)$-th leftmost box of each to be the same, so
  $\mu \cap A_{n+1} = \nu \cap A_{n+1}$, a contradiction.
\end{proof}

\begin{proposition}\label{prop:14}
  Given a staircase Prym tableau $t$, there exists a strip $\mu$ and a
  tableau $s$ that is non-repeating in $\mu$ such that $s$ dominates
  $t$. Moreover, in the even case, this strip may be chosen to be
  horizontal.
\end{proposition}

\begin{proof}
  First suppose that $k$ is odd.  We begin by defining a tableau
  $s_l = t\restrict{T_l}$ and a strip $\mu_l = T_l$, and proceed by
  induction: suppose that we have defined a tableau $s_n$ on $T_n$
  that is non-repeating on a strip $\mu_n \subset T_n$, and suppose
  that $s_n\restrict{\mu_n} = t\restrict{\mu_n}$.  Let $(x,y)$ be the
  $n$-th leftmost box; then $(x+l-1,y-l+1)$ is the $n$-th rightmost
  box.  These two boxes, separated by distance less than $k$, cannot
  contain the same symbol in $t$.  Extend $\mu_n$ to a strip
  $\mu_{n+1} \subset T_{n+1}$ by defining the $(n+1)$-leftmost box to
  be $(x+1,y)$ if $t(x,y) < t(x+l-1,y-l+1)$, and $(x,y+1)$ otherwise.
  (This ensures that $s_{n+1}$, which we shall presently define,
  satisfies the gluing condition.)  Then extend $s_n$ to the map
  $s_{n+1}$ that agrees with $t$ on $\mu_{n+1} \cap A_{n+1}$ and is
  defined elsewhere according to the repeating conditions.  It is not
  hard to see that $s_{n+1}$ is non-repeating in $\mu_{n+1}$ and
  dominates $t\restrict{T_{n+1}}$.  Take $s = s_r$ for the desired
  result.
 
  We now consider the case that $k$ is even.  Given each
  $\omega \in A_n \cap D_{i,k}$, define
  $s(\omega) = \max t(A_n \cap D_{i,k})$.  It is clear that $s$ dominates
  $t$.  We claim that $s$ is non-repeating on the horizontal strip
  $\mu_0$.  Note first that the displacement and tableau conditions
  are satisfied everywhere.  The first follows immediately from the
  definition of $s$; to prove the second, it suffices to show that
  $s(x,y) < s(x+1,y)$ for each box $(x,y)$ since the other inequality,
  $s(x,y) < s(x,y+1)$, follows by transposing coordinates.  Indeed,
  suppose that $(x,y) \in A_n \cap D_{i,k}$.  Choose a box
  $(x',y') \in A_n \cap D_{i,k}$ that satisfies
  $t(x',y') = \max t(A_n \cap D_{i,k})$; then
  \begin{equation*}
    s(x,y) = t(x',y') < t(x'+1,y') \leq \max t(A_{n+1} \cap D_{i+1,k})
           = s(x+1,y).
  \end{equation*}
  The equalities follow from the definition of $s$, the strict
  inequality follows from the tableau condition on $t$, and the weak
  inequality follows because $(x'+1,y') \in A_{n+1} \cap D_{i+1,k}$.
  
  The left repeating condition holds because, given a box $(x,y)$ in
  the left component, $(x,y) \in A_n \cap D_{i,k}$ implies that
  $(x+l,y-l) \in A_n \cap D_{i,k}$, so $s(x,y) = s(x+l,y-l)$.  The right
  repeating condition is vacuously satisfied because the right
  component is empty.  The gluing condition follows from the tableau
  and left repeating conditions: $s(x,l) < s(x,l+1) = s(x+l,1)$.
\end{proof}

The fact that the Prym--Brill--Noether locus is pure-dimensional
readily follows from the results of this section.  

\begin{proposition}\label{thm:pure-dim}
	$V^r(\Gamma,\varphi)$ is pure-dimensional for any gonality $k$.
\end{proposition}

\begin{proof}
  Given a Prym tableau $t$, we want to find a Prym tableau that
  dominates $t$ and has codimension $n(r,k)$.  Indeed, apply the
  reflection algorithm of \cref{prop:reflective} to $t$; the resulting
  tableau $u$ dominates $t$.  In the generic case, $u\restrict{T_r}$
  is injective, so we are done.  Otherwise, apply \cref{prop:14} to
  $u\restrict{T_r}$ to obtain a map $s$ defined on $T_r$ that
  dominates $t\restrict{T_r}$ by transitivity.  This map is a minimal
  tableau by \cref{prop:18}.  Extend $s$ uniquely to a reflective
  tableau.  This is the desired tableau.
\end{proof}

The ultimate motivation for defining non-repeating tableaux comes from
\cref{prop:2,prop:14}, which for fixed parameters $(g,r,k)$ yield the
following powerful correspondence:
\begin{equation}
\label{eq:5}
\set{\text{maximal cells of
    $V^r(\Gamma,\varphi)$}} \leftrightarrow \set{\text{non-repeating tableaux of type $(g,r,k)$}}.
\end{equation}
This result will be invaluable in the remainder of this section and in
\cref{sec:counting}.

For $k \leq 2r$, define a \define{strip tableau of type $(g,r,k)$} to
be an injective tableau $t$ defined on a strip $\mu$ of length $r$ and
width $l$ such that $t$ satisfies the gluing condition on $\mu$ and
takes values in $[g-1]$.\footnote{We require that $k \leq 2r$ because
  otherwise $\mu$, which contains $T_l$ by definition, is larger than
  $T_r$---this does not make sense!  In the case that $k > 2r$, the
  ``strip tableaux'' are really just the standard Young tableaux
  defined on $T_r$.  Observe that \cref{eq:8} still holds in this
  case.}  In the odd case, $\mu$ may take any of $2^{r-l}$ possible
shapes; in the even case, we require that $\mu = \mu_0$.  (We adopt
the convention that $\mu = \mu_0$ whenever we refer to non-repeating
tableaux in the even case.)

Clearly, any strip tableau extends uniquely to a non-repeating tableau
by applying the repeating conditions.  Conversely, any non-repeating
tableau determines a unique strip tableau.  Hence, the two classes of
tableaux are equivalent, and we may use either one depending on the
circumstances.  Keeping in mind the nuances that distinguish the odd
and even cases, \cref{eq:5} yields the correspondence
\begin{equation}
\label{eq:8}
\set{\text{maximal cells of
    $V^r(\Gamma,\varphi)$}} \leftrightarrow \set{\text{strip tableaux of type $(g,r,k)$}}
\end{equation}
as expected.  Just as we extended the definitions pertaining to Prym
tableaux to staircase Prym tableaux (see the discussion at the end of
\cref{sec:reflective}), so too may we extend these definitions even
further to strip tableaux.

%
%
%
%
%

\subsection{Connectedness}\label{sec:connect}

In this section, we shall occupy ourselves with the following result,
which we prove in three seperate cases---generic, even, and odd.

\begin{proposition}\label{thm:path}
  If $\dim V^r(\Gamma,\varphi) \geq 1$, then $V^r(\Gamma,\varphi)$ is
  connected in codimension 1.
\end{proposition}

As we explain in \cref{rem:1}, the following definition captures the
analogous notion of connectedness on the level of tableaux.

\begin{definition}\label{def:3}
  Suppose that $t$ and $s$ are $k$-uniform displacement tableaux
  having the same shape, the same number of distinct symbols, and
  image contained in $[g-1]$ for some $g$.  We say that $t$ and $s$
  are \define{adjacent} if there exist two symbols $a,b \in [g-1]$ and
  two indices $i,j \in \Z/k\Z$ such that
  \begin{itemize}
  \item $s(D_{i,k}) = t(D_{i,k}) \cup \set{a}$,
  \item $t(D_{j,k}) = s(D_{j,k}) \cup \set{b}$,
  \item $t(D_{h,k}) = s(D_{h,k})$ for all $h \in \Z/k\Z \setminus \set{i,j}$, and
  \item $a = b$ only if $i = j$.
  \end{itemize}
  We say that $t$ and $s$ are \define{connected} if there exists a
  sequence $(t_i)_{i=0}^n$ of tableaux such that $t_0 = t$, $t_n = s$,
  and $t_i$ is adjacent to $t_{i+1}$ for each
  $i \in \set{0,\ldots,n-1}$.
\end{definition}

\begin{remark}\label{rem:1}
  Using the tableau--divisor correspondence (\cref{sec:prym-tableaux}), it is
  straightforward to show that, in the particular case where $t$ and $s$ are minimal
  tableaux of type $(g,r,k)$, $t$ and $s$ are adjacent just in case either
  $P(t) \cap P(s)$ is a torus of codimension 1 
  or $P(t) = P(s)$.  The former corresponds to the case $a \neq b$, the latter to the
  case $a = b$.\footnote{If we did not require that $i = j$ whenever $a = b$, then we
    could obtain $P(t) \cap P(s) = \emptyset$, which is undesirable.}  Likewise, the
  tableaux $t$ and $s$ are connected just if $P(t)$ and $P(s)$ are connected in
  codimension 1 by a sequence of cells of $V^r(\Gamma,\varphi)$.  It then follows
  from the correspondence in \cref{eq:5} that in order to prove \cref{thm:path}, it
  suffices to show that any two non-repeating tableaux (of the same type) are
  connected.
\end{remark}

\cref{def:3} is unintuitive and unwieldy, but in fact, there are
relatively straightforward methods by which we may produce connected
tableaux.  As we shall see, to prove that any two non-repeating
tableaux are connected, we only need the following three operations
and combinations thereof.  Fix a $k$-uniform displacement tableau $t$
and a symbol $a \in [g-1]$ that does not appear in $t$.
\begin{enumerate}
\item Choose a box $\omega$ and define $s(\omega) = a$ and $s(\omega') = t(\omega')$
  for all $\omega' \neq \omega$.  We call this procedure \define{swapping $a$ into
    $\omega$}.  In general, $s$ will not satisfy the tableau condition unless $a$ is
  greater than the symbols in the lower neighbors of $\omega$ and smaller than those
  in the upper neighbors.  Moreover, $s$ has the same number of symbols as $t$ just
  in case the original symbol, $t(\omega)$, does not appear elsewhere in $t$.  Given
  that $s$ satisfies these conditions, $s$ is adjacent to $t$.
  
\item To ensure that $s$ does not have more symbols than $t$, we
  introduce the related notion of \define{swapping $a$ in for $b$},
  where $b$ is any symbol.  If $b$ does not appear in the tableau,
  define $s = t$ (i.e., do nothing).  Otherwise, for each box $\omega$
  containing $b$, define $s(\omega) = a$, and for all other boxes
  $\omega'$, define $s(\omega') = t(\omega')$.  The tableau condition
  must again be checked, this time at each box $\omega$.  Supposing
  that it holds, $s$ is adjacent to $t$.
  
\item It is straightforward to check that we may always swap $a$ in
  for $a+1$ and $a$ in for $a-1$.  Hence, if there is a symbol $b > a$
  that we want to pull out of the tableau, we iterate the following
  procedure: after the $i$-th step, $a+i$ is not in the tableau, so
  swap $a+i$ in for $a+i+1$.  After $b-a$ steps, each symbol $a+i$ in
  $t$ for $i \in [b-a]$ has been decremented by 1.  In particular, $b$
  no longer appears in the tableau.  An analogous procedure may be
  used in the case that $b < a$; in either case, we call this
  \define{cycling out $b$ using $a$}.  The resulting tableau $s$ is
  connected to $t$.
\end{enumerate}

The following lemma outlines the first step in proving \cref{thm:path}.

\begin{lemma}\label{lem:18}
  Any two injective tableaux of the same shape containing at most
  $g-2$ symbols are connected.
\end{lemma}

\noindent To prove it, we introduce one more tool.  Given any shape
$\lambda \subset \N^2$, we establish a total order on its boxes as
follows: for any $(x,y) \in A_m$ and $(x',y') \in A_n$, say that
$(x,y) < (x',y')$ if $m < n$, or if both $m = n$ and
$x < x'$.\footnote{When $(x,y) < (x',y')$, we say that $(x,y)$ is
  ``smaller'' than $(x',y')$ in order to avoid confusion with the
  previous terminology ``below'' (see \cref{sec:prym-tableaux}); the
  latter implies the former, but the converse does not hold in
  general.}  Let $Q_{\lambda}(\omega)$ be the place of the box
$\omega$ in the order, i.e., the number of boxes $\omega' \in \lambda$
for which $\omega' \leq \omega$.  Then we define an $\N$-valued
function $R_{\lambda}$ on injective tableaux of shape $\lambda$ so
that
\begin{equation*}
  R_{\lambda}(t) \coloneq \card{\lambda} - \max \set{Q_{\lambda}(\omega)
    \given t(\omega') = Q_{\lambda}(\omega') \text{ for all }
    \omega' \leq \omega}.
\end{equation*}
There is a unique tableau $\bar{t}$ for which
$R_{\lambda}(t) = 0$; call it the \define{standard increasing
  tableau}.  For example, if $\lambda = T_4$, then $\bar{t}$ is the
final tableau in \cref{fig:4}, while the value of $R_{T_4}$ at the
first tableau in the sequence is 8, since the first and second boxes
contain the correct symbols but the third box does not---it contains a
5 rather than a 3.  Intuitively, $R_{\lambda}$ measures how far a
given tableau is from being identical to $\bar{t}$.

\begin{proof}[Proof of \cref{lem:18}]
  We will show by induction on the values of $R_{\lambda}$ that any
  injective tableau $t$ of shape $\lambda$ is connected to the
  standard increasing tableau $\bar{t}$.

  If $R_{\lambda}(t) = 0$, then the statement is trivially true since
  it must be the case that $t = \bar{t}$.  Otherwise, suppose that any
  tableau $s$ with $R_{\lambda}(s) < R_{\lambda}(t)$ is connected to
  $\bar{t}$.  Then it suffices to show that $t$ is connected to some
  such $s$.
  
  Let $\omega$ be the smallest box (relative to the total order) such
  that $t(\omega) \neq Q_{\lambda}(\omega)$.  Denote by $a$ the value
  $Q_{\lambda}(\omega)$ and by $S$ the set of boxes strictly smaller
  than $\omega$.  Our goal is to find a sequence of swapping
  operations that leaves the boxes of $S$ (which already contain the
  correct symbols) untouched while inserting the symbol $a$ into
  $\omega$.  
  Choose the smallest symbol $b$ not in $t$.  Observe that $a \leq b \leq g - 1$; the
  lower bound holds because every symbol less than $a$ appears in (the correct box
  of) the tableau, while the upper bound follows from the assumption that $t$ uses at
  most $g - 2$ symbols.  First, cycle out $a$ using $b$ and call the resulting
  tableau $u$.  This has the effect of incrementing any symbol
  $\set{a,a+1,\ldots,b-1}$ appearing in $t$, so symbols in $S$ are unaffected.
  Furthermore, $u$ is injective and connected to $t$.  Second, swap $a$ into $\omega$
  and call the resulting tableau $s$.  This operation trivially leaves $S$
  unaffected.  Moreover, $s$ satisfies the tableau condition: the fact that
  $a < u(\omega)$ covers the upper neighbors, while the lower neighbors are both in
  $S$ (provided that they are in $\lambda$ at all) and so contain symbols that are
  smaller than $a$.  Since $u$ is injective, $u(\omega)$ appears only once, so $s$
  contains the same number of symbols as $u$.  It follows that $s$ is connected to
  $u$ and hence to $t$ as well by transitivity.  Since
  $R_{\lambda}(s) \leq R_{\lambda}(t) - 1$, this completes the proof.
\end{proof}

\begin{proof}[Proof of \cref{thm:path} in the generic case.]
  In the generic case, the non-repeating tableaux are precisely the
  injective tableaux on $T_r$.  Each such tableau contains
  $g-1-\dim V^r(\Gamma,\varphi)$ symbols; since
  $\dim V^r(\Gamma,\varphi) \geq 1$, we apply \cref{lem:18} to find
  that any two are connected.  By \cref{rem:1}, we are done.
\end{proof}

\begin{example}\label{ex:2}
  In \cref{fig:4}, we outline the proof of \cref{lem:18} in the case
  that $\lambda = T_4$ and $g \geq 12$.  We begin with the tableau on
  the top right and terminate at the standard increasing tableau, each
  step either a cycle or a swap.  The set of boxes $S$ at each step is
  colored blue.  The first step cycles out 3 using 11; notice that
  each symbol greater than or equal to 3 is incremented.  The second
  step swaps 3 into $(2,1)$, thereby removing 6 from the tableau.  We
  continue cycling and swapping as appropriate until every symbol is
  in the correct position according to the order.
  \begin{figure}[htb]
    \input{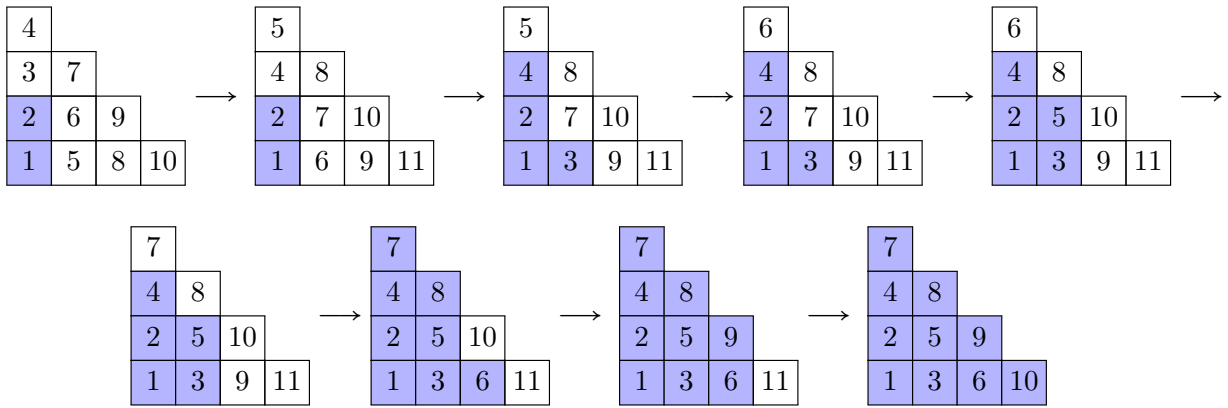}
    \caption{Example application of the algorithm from the proof of
      \cref{lem:18}.}
  \label{fig:4}
  \end{figure}
\end{example}

\begin{lemma}\label{lem:17}
  If $\dim V^r(\Gamma,\varphi) \geq 1$, then for any gonality $k$, any
  two tableaux of type $(g,r,k)$ non-repeating in the horizontal strip
  $\mu_0$ are connected.
\end{lemma}

The proof is very similar to that of \cref{lem:18}, so we highlight
only the major differences.  The key idea is that any tableau
non-repeating in $\mu_0$ is uniquely determined by its restriction to
$\mu_0$.  Hence, we may naturally define $R_{\mu_0}$ on such tableaux,
and we take as our base point $\bar{t}$ the unique tableau which
extends the standard increasing tableau on $\mu_0$.  We also make use
of the following essential observations.

\begin{remark}\label{rem:5}
  Suppose that $t$ is non-repeating in a strip $\mu$ and that the map $s$ is produced
  from $t$ by swapping $a$ in for $b$.  Then $s$ trivially satisfies the repeating
  conditions.  Meanwhile, by \cref{cor:1}, we know that $b$ appears exactly once in
  $t\restrict{\mu}$; provided that $s$ satisfies the tableau and gluing conditions at
  the box containing $b$, it is non-repeating in $\mu$.  Now, the gluing
  condition is satisfied whenever the repeating and tableau conditions are
  (everywhere) satisfied.  Hence, if $s$ is produced from $t$ by cycling out $b$
  using $a$, then $s$ is non-repeating in $\mu$.
\end{remark}

\begin{proof}[Proof of \cref{lem:17}.]
  As before, we induct on the values of $R_{\mu_0}$, noting that
  $\bar{t}$ is the unique tableau satisfying $R_{\mu_0}(\bar{t}) = 0$.
  Let $t$ be a tableau non-repeating in $\mu_0$ and $\omega \in \mu_0$
  the smallest box such that $t(\omega) \neq Q_{\mu_0}(\omega)$; much
  as before, denote $Q_{\mu_0}(\omega)$ by $a$, and let $S$ be the set of
  boxes in $\mu_0$ smaller than $\omega$.  First, cycle out $a$ using
  $b$, where $b \in \set{a,a+1,\ldots,g-1}$ does not appear in $t$ (and
  exists by the assumption that $\dim V^r(\Gamma,\varphi) \geq 1$).
  The resulting tableau $u$ is connected to $t$ and, by \cref{rem:5},
  is non-repeating in $\mu_0$.

  Here we diverge from the proof of \cref{lem:18}: instead of merely swapping $a$
  into the box $\omega$, we swap it in for the symbol $u(\omega)$; call the result
  $s$.  By \cref{rem:5}, to show that $s$ is non-repeating in $\mu_0$, it suffices to
  check the tableau and gluing conditions at $\omega$.  The former follows as in the
  proof of \cref{lem:18}, since $a < u(\omega)$ and the lower neighbors (if they
  exist) are in $S$.  The latter is trickier.  Because $\mu_0$ is horizontal, every
  leftmost box is of the form $(x,l)$ for some $x$.  The gluing condition on $u$ says
  that $u(x,l) < u(x+l-\epsilon,1)$.  If $\omega = (x,l)$, then the gluing condition
  is satisfied for $s$ since $a < u(x,l)$.  If $\omega = (x+l-\epsilon,1)$, then
  $(x,l) < \omega$. Therefore, $(x,l) \in S$ and contains a symbol smaller than $a$.
  In any other case, the gluing condition is trivially satisfied.

  Hence, $s$ is a tableau non-repeating in $\mu_0$ that is connected to $t$ and
  satisfies $R_{\mu_0}(s) \leq R_{\mu_0}(t) - 1$.
\end{proof}

\begin{proof}[Proof of \cref{thm:path} in the even case.]
  The non-repeating tableaux for even values of $k$ are each
  non-repeating in $\mu_0$ in particular (by the discussion following
  \cref{eq:5}).  By \cref{lem:17} and \cref{rem:1}, we are done.
\end{proof}

The odd case is more difficult than the even case because we cannot
just consider the horizontal strip: by the correspondence in
\cref{eq:5}, each of the $2^{r-l}$ strips determines a distinct set of
maximal cells of $V^r(\Gamma,\varphi)$.  Therefore, we introduce a
height function $H$ and---as we did for $R_{\lambda}$---show that any
tableau is connected to another with a lower $H$ value.  Given a
tableau $t$ of odd torsion which is non-repeating in $\mu$, define
$H(t)$ to be the second coordinate of the $r$-th leftmost box of
$\mu$.  Note that $H$ is well-defined by \cref{prop:2}.  Moreover,
$H(t) = l$ if and only if $\mu = \mu_0$.

We again take $\bar{t}$ to be unique non-repeating tableau that
extends the standard increasing tableau on $\mu_0$.  To simplify our
notation, we introduce the unit vectors $\hat x$ and $\hat y$ to
describe boxes relative to other boxes.  For example, if
$\omega = (x,y)$, then $\omega + \hat x = (x+1,y)$ and
$\omega - 2\hat y = (x,y-2)$.

\begin{proof}[Proof of \cref{thm:path} in the odd case.]
  We shall prove it by induction on the values of $H$.  First, suppose
  that $H(t) = l$.  Then $t$ is non-repeating in $\mu_0$, hence
  connected to $\bar{t}$ by \cref{lem:17}.  For the induction step,
  suppose that $t$ is non-repeating in a strip $\mu$ and that every
  tableau $s$ satisfying $H(s) < H(t)$ is connected to $\bar{t}$.
  Denote by $(x,y)$ the unique box in $\mu$ for which $y = H(t)$ and
  $(x-1,y) \nin \mu$.  Denote its anti-diagonal by $A_q$, and let
  $n = r - q$.  Define $\psi_i \coloneq (x+i,y)$ for each
  $i \in \set{0,1,\ldots,n}$.  Since $H(t) = y$, $\psi_i$ is the
  $(q + i)$-th leftmost box for all $i$, and in particular, $\psi_n$
  is the $r$-th leftmost box.

  Our goal is to show that $t$ is connected to a tableau $s$ that is non-repeating in
  $\nu$, where $\nu$ is the strip that agrees with $\mu$ up to $A_{q-1}$ but has
  every subsequent leftmost box east of the previous one.  (In particular, the $q$-th
  leftmost box of $\nu$ is $(x+1,y-1)$ rather than $(x,y)$.)  Then we will be done,
  since $H(s) = H(t) - 1$.

  \textit{Preliminary observations.}  Notice that each $\psi_i$ is in the left
  component of $T_r \setminus \nu$.  Hence, in order for $s$ to satisfy the left
  repeating condition, we need $s(\psi_i) = s(\omega_{i,0})$, where
  $\omega_{i,0} \coloneq (x+i+l-1,y-l)$.  Note that $\omega_{0,0}$ is the $(q-1)$-th
  rightmost box of $\mu$, while for each $i \geq 1$, $\omega_{i,0}$ is in the right
  component of $T_r \setminus \mu$.  Hence, for each $i \geq 1$, the right repeating
  condition yields $t(\omega_{i,0}) = t(\psi_i - \hat x - \hat y)$.  More generally,
  for each $j \geq 0$ we define $\omega_{i,j} = (x+i+l-1+jl,y-l-j(l-1))$.  Then for
  all $i$ and $j \geq 1$, $\omega_{i,j}$ is in the right component of both
  $T_r \setminus \mu$ and $T_r \setminus \nu$, so we have
  $t(\omega_{i,j}) = t(\omega_{i,0})$ and $s(\omega_{i,j}) = s(\omega_{i,0})$ by the
  right repeating condition.  See \cref{fig:pathoddproof} for a schematic diagram of
  our notations.

  \begin{figure}[htb]
    \newcommand*{\scs}{\scriptstyle}
\centering

\begin{tikzpicture}[]
  \node (n)
  {
    \begin{varwidth}{5cm}
      {
        \ytableausetup {boxsize=.625cm}
        \begin{ytableau}
          \scs{\psi_0} & \scs{\psi_1} & \scs{\psi_2} &\none[\dots]
            & \scs{\psi_{n-1}} & \scs{\psi_n} \\
          *(c1) & *(c1) & *(c1) & *(c1) & \none[\dots] & *(c1) & \\
          \none & \none[\originalddots] & \none[\originalddots]
            & \none[\originalddots] & \none[\originalddots]
            & \none[\originalddots] & \none[\originalddots]
            & \none[\originalddots] \\
          \none & \none &  &  &  &  & \none[\dots] & \scs & \\
          \none & \none &\none & \scs{\omega_{0,0}}
          & *(c2)\scs{\omega_{1,0}} & *(c2)\scs{\omega_{2,0}}
          & *(c2)\scs{\omega_{3,0}} & \none[\dots]
          & *(c2)\scs{\omega_{n,0}} & *(c2)
        \end{ytableau}
      }
    \end{varwidth}
  };
  \draw[line width=0.075cm,black] (-2.5,.9)--(-2.5,1.6);
  \draw[line width=0.075cm,black] (-2.5,1.6)--(1.4,1.6);
  \draw[line width=0.075cm,black] (1.4,1.6)--(1.4,1);
  \draw[line width=0.075cm,black] (1.4,1)--(2,1);
  \draw[line width=0.075cm,black] (2,1)--(2,.3);
  \draw[line width=0.075cm,black] (2,.3)--(2.6,.3);
  \draw[line width=0.075cm,black] (2.6,.3)--(2.6,-.3);
  \draw[line width=0.075cm,black] (2.6,-.3)--(3.3,-.3);
  \draw[line width=0.075cm,black] (3.3,-.3)--(3.3,-.95);
  \draw[line width=0.075cm,black] (3.3,-.95)--(0.05,-.95);
  \draw[line width=0.075cm,black] (0.05,-.95)--(0.05,-1.6);
\end{tikzpicture} \ \ \ \ \
\begin{tikzpicture}[]
  \node (n)
  {
    \begin{varwidth}{5cm}
      {
        \ytableausetup {boxsize=.625cm}
        \begin{ytableau}
          *(c2)\scs{\psi_0} & *(c2)\scs{\psi_1} & *(c2)\scs{\psi_2}
            &\none[\dots] & *(c2)\scs{\psi_{n-1}}
            & *(c2)\scs{\psi_n} \\
          &  &  &  & \none[\dots] &  & \\
          \none & \none[\originalddots] & \none[\originalddots] &
          \none[\originalddots] & \none[\originalddots] &
          \none[\originalddots] & \none[\originalddots]
          & \none[\originalddots] \\
          \none & \none &  &  &  &  & \none[\dots] & \scs & \\
          \none & \none &\none & *(c1)\scs{\omega_{0,0}} &
          *(c1)\scs{\omega_{1,0}} & *(c1)\scs{\omega_{2,0}} &
          *(c1)\scs{\omega_{3,0}} & \none[\dots] &
          *(c1)\scs{\omega_{n,0}} &
        \end{ytableau}
      }
    \end{varwidth}
  };
  \draw[line width=0.075cm,black] (-2.5,1)--(2,1);
  \draw[line width=0.075cm,black] (2,1)--(2,.3);
  \draw[line width=0.075cm,black] (2,.3)--(2.6,.3);
  \draw[line width=0.075cm,black] (2.6,.3)--(2.6,-.3);
  \draw[line width=0.075cm,black] (2.6,-.3)--(3.3,-.3);
  \draw[line width=0.075cm,black] (3.3,-.3)--(3.3,-1);
  \draw[line width=0.075cm,black] (3.3,-1)--(3.9,-1);
  \draw[line width=0.075cm,black] (3.9,-1)--(3.9,-1.6);
  \draw[line width=0.075cm,black] (3.9,-1.6)--(-.6,-1.6);
\end{tikzpicture} \ \ \ \ \
    \caption{The left and
      right tableaux are restrictions of $t$ and $s$ respectively to
      the same subset of $T_r$.  The rightmost anti-diagonal in the
      subset is $A_r$.  The bold lines outline the strips
      $\mu$ and $\nu$ respectively.  The labels are box names,
      not symbols; only the $\psi_i$ and $\omega_{i,0}$ are shown.
      Symbols in the red boxes outside of each strip are copied from
      the respective blue boxes within the strip.
    }
    \label{fig:pathoddproof}
  \end{figure}
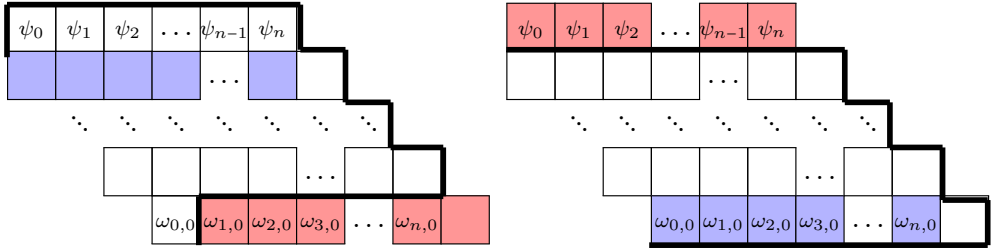

  \textit{First attempt, using $t$.}  To go from $t$ to $s$, we could try to replace
  the symbol in $\omega_{i,j}$ with the symbol in $\psi_i$ for each $i$ and $j$, and
  leave all other symbols unchanged.  Unfortunately, the tableau condition would
  necessarily fail at $\omega_{n,0}$, which lies on $A_{r-1}$.  Indeed, we have
  $t(\psi_n) > t(\psi_n - \hat y) = t(\omega_{n,0}+ \hat x)$.  It is also possible
  that $t(\psi_n) > t(\omega_{n,0} + \hat y)$.  We shall modify $t$ so that these
  issues are avoided; in particular, we will put the two largest symbols, $g-2$ and
  $g-1$, into $\omega_{n,0} + \hat y$ and $\omega_{n,0} + \hat x$, respectively.

  \textit{Construction of $u$.}  Cycle out $g-2$ using any symbol that does not
  appear in $t$.  By \cref{rem:5}, the resulting tableau is still non-repeating in
  $\mu$.  The only symbol greater than $g-2$ is $g-1$, and if that appears in the
  tableau, then it appears on $A_r$ (since, being the largest symbol, it cannot be in
  a box with upper neighbors).  Hence, when we swap $g-2$ into
  $\omega_{n,0} + \hat y$, we know that the tableau condition is satisfied since it
  is necessarily larger than both of its lower neighbors' symbols.  Moreover, the
  symbol it replaces is unique in the tableau; indeed, any symbol in $A_r \cap \mu$
  in a tableau non-repeating in $\mu$ is unique.  Thus, the resulting tableau is
  still non-repeating in $\mu$.  Next, cycle out $g-1$ (using whatever symbol is
  free) and call the resulting tableau $v$.  Again, $v$ is non-repeating in $\mu$ and
  connected to $t$ by the previous operations.

  Now we swap $g-1$ into $\omega_{n,0} + \hat x$ to produce a $k$-uniform
  displacement tableau $u$.  Observe that $u$ is not non-repeating in $\mu$; indeed,
  $\omega_{n,0} + \hat x$ is in the right component of $T_r \setminus \mu$, and the
  box it should be repeated from, $\psi_n - \hat y$, contains a symbol other than
  $g-1$ (which we had cycled out).  It is important to note that, as a result, the
  codimension of $P(u)$ relative to $V^r(\Gamma,\varphi)$ is 1.  Thus, we need to
  show that $u$, which is dominated by $v$, is also dominated by some tableau $s$
  non-repeating in $\nu$; this will imply that $t$ and $s$ are adjacent.

  \textit{Second attempt, using $u$.}  We now construct $s$ from $u$ in the same way
  that we attempted to construct $s$ from $t$.  Precisely, for each $i$ and $j$, we
  define $s(\omega_{i,j}) = u(\psi_i)$ and let $s$ coincide with $u$ everywhere else.
  It is not hard to see that $s\restrict{\nu}$ is injective and so trivially
  satisfies the displacement condition.  Moreover, the left and right repeating
  conditions are satisfied since $s(\psi_i) = s(\omega_{i,0})$ and
  $s(\omega_{i,j}) = s(\omega_{i,j-1})$ for each $j \geq 1$.  Moreover, the leftmost
  boxes of $\nu$ that are not leftmost boxes of $\mu$ are of the form
  $\psi_i - \hat{y}$ for each $i$.  The gluing condition on $s$ then follows by the
  tableau condition on $u$, since the $\omega_{i,0}$ are the corresponding rightmost
  boxes of $\nu$; explicitly,
  \[s(\psi_i - \hat{y}) = u(\psi_i - \hat{y}) < u(\psi_i) = s(\psi_i)
    = s(\omega_{i,0}).\]

  It remains to show that $s\restrict{\nu}$ satisfies the tableau condition.  This
  amounts to checking it at each $\omega_{i,0}$.  The south neighbor of
  $\omega_{i,0}$ is not in $\nu$, so we may safely ignore it.  Moreover, since the
  entire block of symbols in the set $\set{\omega_{i,0}}$ is copied from
  $\set{\psi_i}$, we know that the condition is satisfied between each pair
  $\set{\omega_{i,0}, \omega_{i+1,0}}$.  For $i \neq n$, we need to check that the
  north neighboring symbol is larger:
  \[s(\omega_{i,0}) = u(\psi_i) < u(\psi_i + \hat y) = u(\omega_{i,0} + \hat y) =
    s(\omega_{i,0} + \hat y).\] For $i = 0$, the west must be smaller:
  \[s(\omega_{0,0}) = u(\psi_0) > u(\psi_0 - \hat y) = u(\omega_{1,0}) >
    u(\omega_{0,0} - \hat x) = s(\omega_{0,0} - \hat x).\] Finally, for $i = n$, the
  upper neighbors' symbols are $g-1$ and $g-2$, the two largest symbols in the
  tableau.  Hence, $s$ is non-repeating in $\nu$.  Moreover, it clearly dominates
  $u$, which completes the proof.
\end{proof}

The next example demonstrates the algorithm that lowers the height of
the strip by one.

\begin{example}\label{ex:3}
  Consider the first tableau $t$ in \cref{fig:path-connected-odd-ex},
  where $g=23$, $r=8$, and $k=5$. We color the strip $\mu$ blue. In
  our example, we note that $H(\mu)=5$, so $\psi_0=(3,5)$.  The first
  step of the algorithm is to cycle out $g-2=21$ using $22$ and then
  swap it into $\omega_{n,0}+\hat{y}=(6,3)$.  (These two operations do
  not change the tableau, since $21$ was already in the correct box.)
  The next step is to cycle out $g-1=22$ (this does nothing) and swap
  it into $\omega_{n,0}+\hat{x}=(7,2)$, thereby producing the second
  tableau $u$.  Since $u$ is not minimal, we do not color any strip.
  For the last step, we copy the symbols from the boxes $\psi_0$ and
  $\psi_1$ (17 and 19 respectively) into the boxes $\omega_{0,0}$ and
  $\omega_{1,0}$. This produces the third tableau $s$.  We color the
  strip $\nu$; observe that the height has decreased by one. Another
  iteration of this process would yield a tableau that is
  non-repeating in the horizontal strip.
  \begin{figure}[htb]
    \input{figures/path-connected-odd-ex.tex}
    \caption{}
    \label{fig:path-connected-odd-ex}
  \end{figure}
\end{example}

\section{Enumerative properties}\label{sec:counting}

Now that we have established a few general facts about
Prym--Brill--Noether loci, we begin to look at some of their
enumerative properties.
We start by counting the number of divisors in 0-dimensional loci
before examining the 1-dimensional case.

\subsection{Cardinality of finite Prym--Brill--Noether loci}\label{sec:finite}

In this section, we fix the parameters $g$, $r$, and $k$ so that
$g-1=n(r,k)$.  By \cref{thm:tropicalPBN}, this condition ensures that
$\dim V^r(\Gamma,\varphi) = 0$, so every point of $V^r(\Gamma,\varphi)$ is in itself a maximal cell.

Recall that, by \cref{eq:8}, the maximal cells of the
Prym--Brill--Noether locus 
are in bijection with strip tableaux of the corresponding type.  It is
clear that there are finitely many strips and finitely many ways to
fill each one, so the cardinality of $V^r(\Gamma,\varphi)$, which we
denote by $C(r,k)$, is finite.  This number has been computed
\cite[Corollary~6.1.5]{len2019skeletons} for generic edge lengths
(where $k = 0$ by convention) and equivalently for $k>2r-2$ using the
hook-length formula.  We now compute it in the case that $k$ is even
and at most $2r-2$.




\begin{proposition}\label{prop:card}
  For even $k\leq2r-2$, the number of divisor classes in the 0-dimensional
  locus is
  \begin{equation} \label{eq:3} C(r,k) = n!\sum
    \begin{vmatrix}
      \frac{1}{(r+\alpha_1 k)!} & \frac{1}{(r-2+\alpha_2 k)!} & \cdots
      & \frac{1}{(r-k+2+\alpha_l k)!}\\
      \frac{1}{(r+1+\alpha_1 k)!} & \frac{1}{(r-1+\alpha_2 k)!} &
      \cdots
      & \frac{1}{(r-k+3+\alpha_l k)!}\\
      \vdots & \vdots & \ddots & \vdots\\
      \frac{1}{(r+l-1+\alpha_1 k)!} & \frac{1}{(r+l-3+\alpha_2 k)!}
      & \cdots & \frac{1}{(r-l+1+\alpha_l k)!}\\
    \end{vmatrix}
    \end{equation}
    where $n=n(r,k)=g-1$
    and the sum is taken over all $l$-tuples $(\alpha_i)_{i=1}^l$ for
    which $\alpha_i \in \Z$ and $\sum_{i=1}^l\alpha_i=0$.
\end{proposition}

\begin{proof}
  Using the correspondence in \cref{eq:8} and the definition of a
  strip tableau for $k$ even, we find that $C(r,k)$ equals the number
  of ways to fill out the horizontal strip $\mu_0$ of length $r$ and
  width $l$ using each symbol in $[g-1]$ exactly once while adhering
  to the tableau and gluing conditions.  We aim to construct a
  bijection between these strip tableaux and lattice paths in $\Z^l$
  joining $(l,l-1,\ldots,1)$ to $(r+l,r+l-2,\ldots,r-l+2)$ such that
  each step is in a positive unit direction and every point
  $(z_1,z_2,\ldots,z_l)$ satisfies the constraints
  $z_1>z_2>\cdots>z_l>z_1-k$.  Once we have this, we are done: by
  \cite[Theorem~10.18.6]{bona2015handbook}, the number of such lattice
  paths is exactly given by \cref{eq:3}.

  Given a strip tableau $t$, we obtain a lattice path in the following
  way.  The path begins at $(l,l-1,\ldots,1)$.  Suppose that the first
  $a-1$ steps in the path have been defined and satisfy the conditions
  above.  Identify the unique box $(x,y) \in \mu_0$ containing the
  symbol $a$.  Then define the $a$-th step of the path to be a
  positive unit step in the $y$-th coordinate.


  By the tableau condition, there are precisely $x-1$ values of
  $i \in \set{\ldots,-1,0,1,\ldots,x-1}$ for which $t(x-i,y) < a$
  (namely, the positive ones); this implies that the $a$-th step of
  the lattice path is the $x$-th step in the $y$-th coordinate.
  Applying similar reasoning to the boxes $(x,y-1)$ and $(x,y+1)$,
  which contain symbols less than $a$ and greater than $a$
  respectively, we may conclude that, by the $a$-th step, at least $x$
  steps have been taken in the $(y-1)$-th coordinate and at most $x-1$
  steps have been taken in the $(y+1)$-th coordinate.  Since the
  initial point of the path satisfies $z_{y-1} > z_y > z_{y+1}$, it
  follows that the $a$-th point does as well.  By induction, every
  point in the path satisfies the constraints $z_1>z_2>\cdots>z_l$.


  Next, the gluing condition forces $t(x+l,1)>t(x,l)$ for each $x$.
  On the lattice path, this means that the $(x+l)$-th step in the
  first coordinate must come after the $x$-th step in the $l$-th
  coordinate. At the starting point, the first coordinate is already
  greater by  $l-1$ compared to the $l$-th coordinate, and the gluing condition
  allows this gap to grow to at most $k-1$, giving us the final
  inequality $z_l>z_1-k$.

  Counting the number of boxes in each row of the strip demonstrates
  that the endpoint is $(r+l,r+l-2,\ldots,r-l+2)$, as expected.
  Hence, the procedure defined above, in fact, yields a lattice path of
  the desired form.  Conversely, given a lattice path, we may reverse
  the construction to get a strip tableau: if the $a$-th step in the
  lattice path is the $x$-th step in the $y$-th coordinate, then the
  symbol $a$ goes into box $(x,y)$.  The first $l-1$ inequalities on
  the coordinates verify the tableau condition and the final
  inequality verifies the gluing condition.
  %
  %
\end{proof}


For convenience, we include $C(r,k)$ for small values of $r$ and $k$ in \cref{figure:keven}.

\begin{figure}[htb]
  \centering
  \begin{tabular}{c|c c c c c}
  $r$ & $C(r,0)$ & $C(r,2)$ & $C(r,4)$ & $C(r,6)$ & $C(r,8)$\\\hline
  1 & 1 & 1 & 1 & 1 & 1\\
  2 & 2 & 1 & 2 & 2 & 2\\
  3 & 16 & 1 & 4 & 16 & 16\\
  4 & 768 & 1 & 8 & 128 & 768\\
  5 & 292864 & 1 & 16 & 1024 & 35480\\
  6 & 1100742656 & 1 & 32 & 8178 & 1671168

  \end{tabular}
  \caption{$C(r,k)$ for several values of $r$ and $k$.}
  \label{figure:keven}
\end{figure}


\begin{example}
\label{ex:k=2or4}
  For low values of $k$, we may exhibit all the horizontal strips
  directly. For instance, we claim that $C(r,2)=1$ for every $r$.
  Indeed, the Prym tableaux with minimal codimension are uniquely
  determined by the bottom row, and the only way to fill out the
  row is by using the symbols 1 through $g-1$ in increasing order.

  To see that $C(r,4)=2^{r-1}$, we note first that the tableau
  condition forces the symbol $1$ to be placed into the box $(1,1)$.
  The tableau and gluing conditions together force each subsequent
  pair of symbols $\set{2n-2, 2n-1}$ for $n \in \set{2,3,\ldots,r}$ to
  be placed into $A_n \cap \mu_0$, which contains two boxes.  Thus,
  each value of $n$ yields 2 possibilities for symbol placement.
  This choice is independent of previous choices, so the total number
  of possibilities is $2^{r-1}$.
 	
 	 
\end{example}

When $g-1 > n(r,k)$, the Prym--Brill--Noether locus has positive
dimension.  Its maximal cells still correspond to strip tableaux, but
in this case, each tableau uses only $n(r,k)$ of the $g - 1$
available symbols.  Keeping in mind that $C(r,k)$ counts the number of
strip tableaux when the set of symbols is fixed and every symbol must
be used, we easily obtain the following result.


\begin{proposition}
  \label{prop:numcomp}
  The number of maximal cells of $V^r(\Gamma,\varphi)$ equals
  \begin{equation*}
    C(r,k) \cdot \binom{g-1}{n(r,k)}.
  \end{equation*}
\end{proposition}

$C(r,k)$ remains unknown in the case that $k$ is odd and at most
$2r-3$.  The difficulty lies in counting the number of ways to fill
strips that are not horizontal, which is not in general equal  to $C(r,k)$.

\subsection{First Betti number of 1-dimensional loci}
We now choose $g$, $r$, and $k$ so that $g-1=n(r,k)+1$.
Then $\dim V^r(\Gamma,\varphi) = 1$.   
In particular, the
Prym--Brill--Noether locus is a metric graph that consists of finitely many
circles. Each circle corresponds to a
strip tableau that uses all but a single symbol $a$, which we call the \define{free} symbol.
The circle consists of divisors with a fixed chip on each loop, except for $\ti\gamma_a$ and  $\ti\gamma_{2g-a}$. We refer to these loops as free as well. 

The only way that two different circles intersect is if they have
different free loops and agree on the fixed location of the chips on the other loops.  We also see that if two circles
intersect, then they do so at exactly one point.  It follows that
$V^r(\Gamma,\varphi)$ has a $4$-regular model. Since the 
graph is $4$-valent, the number of edges $e$ equals twice
the number of vertices $v$. The Betti number is therefore
\[
e-v+1 = 2v-v+1=v+1.
\] 
In terms of strip
tableaux, $t$ and $t'$ with free symbols $a$ and $a'$ respectively
give rise to non-trivially intersecting circles precisely when
$a \neq a'$ and $t'$ is obtained from $t$ by swapping $a$ in for
$a'$.





The rest of the section is devoted to calculating  the Betti number of this graph in the generic case and when $k$ is 2 or 4.  We begin with the generic case. 

\genericdimone*

\begin{proof}
  Since the edge lengths are generic, we have a correspondence between maximal cells
  of $V^r(\Gamma,\varphi)$ and injective tableaux defined on $T_r$.
  We know from \cref{thm:tropicalPBN} that $n(r,k) = \binom{r+1}{2}$.
  Since $g - 1 = n(r,k) + 1$, it follows from \cref{prop:numcomp}  that the
  number of maximal cells is
  $C(r,0) \cdot \left(\binom{r+1}{2}+1\right)$.

  The vertices of the 4-regular model of $V^r(\Gamma,\varphi)$ are
  precisely the intersection points between circles. 
  Let $E^{T_r}$ denote the average number of intersection
  points on each circle.  Then the total number of vertices is given
  by $\frac{1}{2} E^{T_r} \cdot C(r,0) \cdot \left(\binom{r+1}{2}+1\right)$.
  Notice the similarity to the first term in \cref{eq:9}; since the Betti number
   is given by $v+1$ (where $v$ is the number of vertices), it suffices to show that $E^{T_r} = r$.
  
  For any skew shape $\lambda$,
    denote by $f^\lambda$ the number of
  distinct injective tableaux defined on $\lambda$ that take values in
  $[n]$, where $n$ is the number of boxes in $\lambda$.  From
  \cite[Theorem~2.9]{chan2018genera}, it follows that the average
  number of intersection points per circle is
  \begin{equation}\label{eq:11}
    E^{\lambda} \coloneq 2\left(r+\sum_{i=1}^{r}\frac{r-i}{n+1} \cdot
      \frac{f^{\prescript{i}{}{\lambda}}}{f^{\lambda}}
      -\sum_{i=1}^{r}\frac{r+1-i}{n+1} \cdot
      \frac{f^{\lambda^{i}}}{f^{\lambda}}\right),
  \end{equation}
  where the terms $\prescript{i}{}{\lambda}$ and $\lambda^{i}$
  describe the tableaux obtained by adding a box to the left or the
  right respectively in the $i$-th
  row.  Taking
  $\lambda = T_r$, $f^{\lambda}$ reduces to $C(r,0)$.
  
	
  Consider the first summation in \cref{eq:11}.  When $i = r$, the
  term is clearly 0.  Provided that $i\neq r$, the resulting shape of
  $\prescript{i}{}{\lambda}$ is not a skew tableau, so
  $f^{\prescript{i}{}{\lambda}}=0$.  Thus, this summation vanishes.
	
  \begin{figure}[htb]
    \centering
    \ytableausetup{boxsize=normal}
    \begin{ytableau}
      1\\
      3 & 1\\
      5 & 3 & 1\\
      7 & 5 & 3 & 1\\
      9 & 7 & 5 & 3 & 1
    \end{ytableau}
    \quad
    \begin{ytableau}
      1\\
      3 & 1\\
      6 & 4 & 2 & 1\\
      7 & 5 & 3 & 2\\
      9 & 7 & 5 & 4 & 1
    \end{ytableau}
    \caption{Hook lengths of each box in $T_5$ and in
      $(T_5)^3$.}\label{fig:5}
  \end{figure}
	
  Next, we look at the second summation.
  We need to enumerate the tableaux obtained by adding a box to the
  end of each row of $T_r$.  Each $f^{\lambda^i}$ can be computed
  using the hook length formula.  We note that in $T_r$, fixing
  $q \in [r]$, the boxes in $A_q$ each have hook length $2(r-q)+1$.
  When a box is added, the hook length of every box in its row and
  column increases by $1$, while all other hook lengths remain the
  same.  (See \cref{fig:5} for an example.)  Thus, the fraction
  $f^{\lambda^{i}}/(n+1)f^{\lambda}$ simplifies down to the ratio of
  the differing hook lengths:
  \begin{gather*}
    \frac{f^{\lambda^{i}}}{(n+1)f^{\lambda}} = \frac{(n+1)!\prod
      h_{\lambda}(i,j)}{(n+1)n!\prod h_{\lambda^i}(i,j)} =
    \frac{(2(r-i)+1)!!(2i-3)!!}{(2(r-i+1))!!(2i-2)!!},
  \intertext{where $h_{\lambda}(i,j)$ is the hook length of box $(i,j)$ in shape $\lambda$ and $(-1)!!$ is defined as 1. We observe that}
    \frac{(2i-3)!!}{(2i-2)!!}=\frac{(2i-3)!!}{2^{i-1}(i-1)!}
    =\frac{(2i-2)!}{2^{i-1}(i-1)!(2i-2)!!}
    =\frac{(2i-2)!}{2^{i-1}(i-1)!2^{i-1}(i-1)!}
    =\frac{\binom{2i-2}{i-1}}{2^{2(i-1)}}.
  \intertext{A similar calculation gives us}
    \frac{(2(r-i)+1)!!}{(2(r-i+1))!!} =
    \frac{\binom{2(r-i+1)}{r-i+1}}{2^{2(r-i+1)}},
  \intertext{so}
    \frac{f^{\lambda^{i}}}{(n+1)f^{\lambda}}
    =\frac{\binom{2i-2}{i-1}\binom{2(r-i+1)}{r-i+1}}{2^{2r}}.
  \intertext{Setting $j=r-i+1$, the sum becomes}
    \sum_{i=1}^{r}(r-i+1) \cdot
    \frac{\binom{2i-2}{i-1}\binom{2(r-i+1)}{r-i+1}}{2^{2r}} =
    \sum_{j=1}^{r}j\cdot
    \frac{\binom{2j}{j}\binom{2(r-j)}{r-j}}{2^{2r}}.
  \intertext{For each $j$ we have}
    j\cdot\frac{\binom{2j}{j}\binom{2(r-j)}{r-j}}{2^{2r}} +
    (r-j)\cdot\frac{\binom{2(r-j)}{(r-j)}
      \binom{2(r-(r-j))}{r-(r-j)}}{2^{2r}} =
    r\cdot\frac{\binom{2j}{j}\binom{2(r-j)}{r-j}}{2^{2r}}.
  \intertext{Thus, grouping $j$ and $r-j$ together, and adding $j=0$ to match $j=r$, we get}
    \sum_{j=0}^{r}j\cdot\frac{\binom{2j}{j}\binom{2(r-j)}{r-j}}{2^{2r}}
    = \frac{r}{2}\sum_{j=0}^{r}
    \frac{\binom{2j}{j}\binom{2(r-j)}{r-j}}{2^{2r}}.
  \end{gather*}
  Finally, by \cite{CBCC}, the sum is equal to 1, so the entire term is
  equal to $\frac{r}{2}$. Plugging this value back into \cref{eq:11},
  we conclude that the average number of vertices on each circle is
  $r$, as desired.
\end{proof}

We conclude the paper by computing the first Betti number of the
Prym--Brill--Noether curve for low even gonality.
\begin{proposition}\label{prop:k2dim1}
  Suppose that $k=2$ and that the Prym--Brill--Noether locus is $1$-dimensional. Then it contains  $r+1=g-1$ circles, and has first Betti number $r+1$.
\end{proposition}

\begin{proof}[Proof.]
  In this case, each tableau contains $g-2$ symbols and is determined by the bottom $1\times r$ rectangle; the positions of the symbols in the strip are determined after choosing a symbol to leave out.  When 1 or $g-1$ is the free symbol, it may only swap into the first or last box in the strip,   respectively, so the corresponding circle only has a single vertex. If any other symbol $m$ is left out,  it can swap with either the symbol $m-1$ or $m+1$, so the corresponding circle has two vertices. Thus, the locus is a chain of $r+1=g-1$ circles wedged together, which has Betti number of $r+1$.
\end{proof}

The last case that we deal with is $k=4$. 

\begin{proposition}
  \label{prop:k4dim1}
  Suppose that $k=4$ and that the Prym--Brill--Noether locus is $1$-dimensional. Then it has the following structure. 
  \begin{enumerate}[label=(\roman*)]
      \item The circles corresponding to the free symbol $1$ have a single vertex.
      \item The circles corresponding to any other odd free symbol have two vertices.
      \item The circles corresponding to the free symbol 2 have three vertices.
      \item The circles corresponding to the free symbol $2r$ have two vertices.
      \item The circles corresponding to any other even free symbol have four vertices.
  \end{enumerate}
  The graph has $2^{r-1}\cdot 2r$ circles and first Betti number  $2^{r-1}(3r-2)+1$.
\end{proposition}

\begin{proof}
 Since $k=4$, the  genus and rank are related by $g=2r+1$. From the gluing condition, it follows that the pair of symbols in each of the boxes $(m+1,1)$ and $(m,2)$ is strictly bigger than the pair of symbols in $(m,1)$ and $(m-1,2)$ (see \cref{tableau:4swap}). In total, for any missing symbol there are $2^{r-1}$ tableaux (see \cref{ex:k=2or4}), giving rise to $2^{r-1}\cdot (2r)$ circles. 
 
 Next, we calculate the number of vertices in the graph, by finding the number of ways of swapping in a free symbol.
  If the free symbol is $1$, it may only be swapped with $2$, which must be in the bottom left corner. Therefore, any circle corresponding to a tableau with missing symbol $1$ has exactly one vertex. Similarly, a missing $2$ may only be swapped for the first three boxes, and a missing $2r$ may only be swapped for the two rightmost boxes.
  
Suppose that the strip is missing an even symbol $2<2m<2r$. Then the symbols in the boxes $(m+1,1)$ and $(m,2)$ are $2m-2$ and $2m-1$, and the symbols in the boxes to to right are $2m+1$ and $2m+2$. The symbol $2m$ may be swapped in for  any of them. 
If, on the other hand, the strip is missing the odd symbol $2<2m+1<2r$, then the boxes  $(m+1,1)$ and $(m,2)$ are $2m$ and $2m+2$, and the symbols to the right are $2m+3$ and $2m+4$. Our symbol $2m+1$ may only be swapped in for of $2m$ or $2m+2$ without violating either the tableau or gluing condition. 

Altogether, we see that there are 
  \begin{gather*}
\frac{(4(r-2) + 2(r-1) + 1 + 3 + 2) \cdot 2^{r-1}}{2} = 2^{r-1}(3r-2)
\end{gather*}
vertices, so the Betti number  is $2^r\cdot (3r-2) + 1$.
\end{proof}

\begin{figure}[htb]
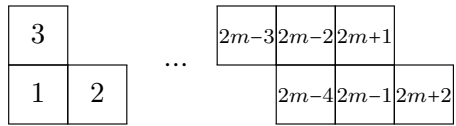

	\centering
	\ytableausetup{mathmode, boxsize=2em}
	\begin{ytableau}
		3\\
		1 & 2
	\end{ytableau}
	\quad$\cdots$\quad
	\begin{ytableau}
		\scriptstyle 2m-3 & \scriptstyle2m-2 & \scriptstyle 2m+1\\
		\none & \scriptstyle 2m-4 &  \scriptstyle 2m-1 & \scriptstyle 2m+2
	\end{ytableau}
	\caption{The bottom 2 rows of a tableau. The symbol $2m$ is missing, and may be swapped in four different boxes.}
	\label{tableau:4swap}
\end{figure}

\begin{example}
  \begin{figure}[htb]
  \centering
  \begin{subfigure}{.5\textwidth}
    \centering
    \includegraphics[scale=0.2]{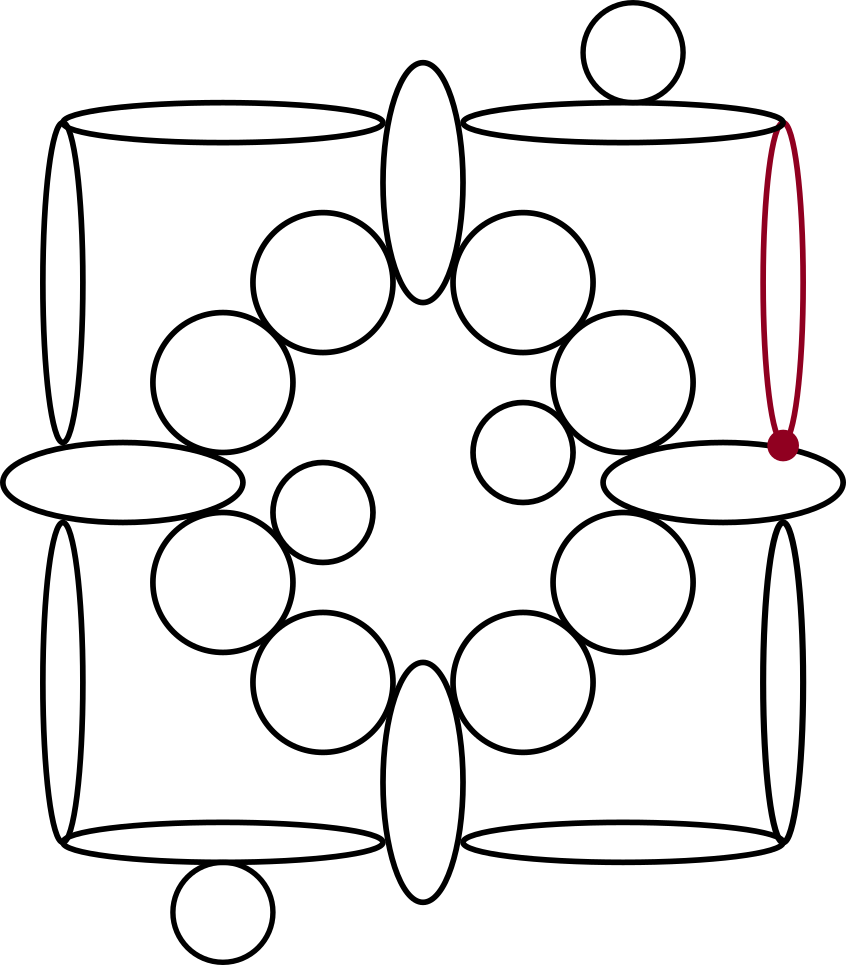}
    \caption{$V^r$ for $(g,k,r)=(7,4,3)$.}
    \label{fig:k4d1}
  \end{subfigure}%
  \begin{subfigure}[t]{.5\textwidth}
    \centering
    \ytableausetup{mathmode, boxsize=2em}
      \begin{ytableau}
      5\\
      4 & 6\\
      1 & 2 & 5
      \end{ytableau}
    \caption{The tableau corresponding to the highlighted circle in the locus.}
  \end{subfigure}
  \caption{}
  \end{figure}

Let $g=7$, $k=4$, and $r=3$. The Prym--Brill--Noether locus is depicted in \cref{fig:k4d1}.
In this case, $n(r,k)=5$, and $C(r,k)=2^{3-1}=4$. \cref{prop:numcomp} shows that the locus consists of $4\cdot\binom{6}{5}=24$ circles, and \cref{prop:k4dim1} implies that the Betti number  is $4(3(3)-2)+1=29$. 

Each of the four circles with $4$ vertices corresponds to a tableau with free symbol 4. The four circles with only a single vertex correspond to the free symbol 1, and the circles they intersect with correspond to the free symbol $2$. The highlighted circle in red is the circle corresponding to the tableau on the right, which has free symbol 3. The highlighted point of intersection corresponds to swapping the symbols $4$ and $3$. 
\end{example}

\bibliographystyle{alpha}
\bibliography{references}

\end{document}